\documentclass[11pt,a4paper]{amsart}

\usepackage{amssymb}
\usepackage{amscd}
\usepackage{pstricks}
\usepackage{pst-node}
\usepackage{geometry}
\usepackage{graphicx}
\usepackage{amsfonts}
\usepackage{latexsym}
\usepackage{amsmath}
\usepackage{amsthm}

\usepackage{geometry}
\newtheorem{thm}{Theorem}[section]

\newtheorem{lem}[thm]{Lemma}
\newtheorem{prop}[thm]{Proposition}

\theoremstyle{definition}

\newtheorem{defn}[thm]{Definition}

\newtheorem{rem}[thm]{Remark}

\numberwithin{equation}{section}

\newcommand{\trans}[1]{{}^t\kern-.2em{#1}}
\newcommand{\ytrans}[1]{{}^t\kern-.11em{#1}}
\newcommand{\Trans}[1]{{}^T\kern-.2em{#1}}
\newcommand{\lsup}[2]{{}^{#1}\kern-.1em{#2}}

\newcommand{\M}{\mathbf{\M}}

\renewcommand{\tilde}[1]{\widetilde{#1}}

\DeclareFixedFont{\bgn}{OT1}{cmr}{m}{n}{20.74}
\DeclareFixedFont{\bgi}{OT1}{cmr}{m}{it}{20.74}

\newcommand{\bigzerou}{\smash{\lower1.7ex\hbox{\bgi O}}}

\makeatletter
\def\eqnarray{%
\stepcounter{equation}%
\def\@currentlabel{\p@equation\theequation}%
\global\@eqnswtrue
\m@th
\global\@eqcnt\z@
\tabskip\@centering
\let\\\@eqncr
$$\everycr{}\halign to\displaywidth\bgroup
\hskip\@centering$\displaystyle\tabskip\z@skip{##}$\@eqnsel
&\global\@eqcnt\@ne \hfil$\displaystyle{{}##{}}$\hfil
&\global\@eqcnt\tw@ $\displaystyle{##}$\hfil\tabskip\@centering
&\global\@eqcnt\thr@@ \hb@xt@\z@\bgroup\hss##\egroup
\tabskip\z@skip
\cr}
\makeatother

\makeatletter
\def\varin{\mathrel{\mathpalette\@varin\relax}}
\def\@varin#1{%
\hbox{\setbox\z@\hbox{\m@th$#1\cup$}%
\def\reserved@a{bold}%
\dimen@\ifx\reserved@a\math@version .3\else .2\fi\p@
\kern.5\wd\z@\kern-\dimen@
\vrule\@width2\dimen@\@height1.08\ht\z@\@depth\z@
\kern-\dimen@\kern-.5\wd\z@
\box\z@}}
\makeatother
\begingroup
\divide\count0 by 60
\multiply\count0 by 60
\advance\count1 by -\count0
\ifnum\count2>11
\ifnum\count2>12 \advance\count2 by -12\fi
\def\ampm{\,pm}%
\else
\ifnum\count2=0 \advance\count2 by 12\fi
\def\ampm{\,am}%
\fi
\xdef\daytime{%
\ifnum\count2<10 0\fi \the\count2:%
\ifnum\count1<10 0\fi \the\count1
\ampm
}%
\xdef\Daytime{
\the\count2:
\ifnum\count1<10 0\fi \the\count1 
\ampm 
}
\endgroup

\begin{document}

\noindent
\thanks{{\bf keywords:} {\it 
    Non-holonomic sub-bundle, 
contact structure, principal bundle, horizontal subspace}
\vspace{1mm}\\ \hspace{2ex}
{\bf MSC 2010}: 53C17, 57S25}

\title[Gromoll-Meyer exotic sphere]
{A co-dimension $3$ sub-Riemannian structure on Gromoll-Meyer exotic sphere}

\author{Wolfram Bauer, Kenro Furutani and Chisato Iwasaki}
\thanks{
The second named author was supported by the Grant-in-aid for Scientific
Research (C) No. 26400124, JSPS and the
National Center for Theoretical Science, National Taiwan University, Taiwan.
The third named author was supported by 
the National Center for Theoretical Science, National Taiwan University, Taiwan.}

\address{Wolfram Bauer (corresponding author)\endgraf
Institut f\"{u}r Analysis, Leibniz Universit\"{a}t  \endgraf
Welfengarten 1, 30167 Hannover, Germany \endgraf
mobil: +49 (0)177 - 6115003 \endgraf
office: +49 (0)511 -  762 2361\endgraf
}
\email{bauer@math.uni-hannover.de}

\address{Kenro Furutani\endgraf
Department of Mathematics\endgraf
Tokyo University of Science\endgraf
2641 Yamazaki\endgraf 
Noda, Chiba, Japan (278-8510) \endgraf } 
\email{furutani$\_$kenro@ma.noda.tus.ac.jp}

\address{Chisato Iwasaki\endgraf
Department of Mathematics \endgraf 
School of Science \endgraf 
University of Hyogo \endgraf 
2167 Shosha Himeji, Hyogo, Japan (671-2201) \endgraf
}
\email{iwasaki@sci.u-hyogo.ac.jp}

\bigskip

\begin{abstract}
We construct a co-dimension $3$ completely non-holonomic sub-bundle on the Gromoll-Meyer exotic $7$ sphere based on 
its realization as a base space of a Sp(2)-principal bundle with the structure group Sp(1).
The same method is valid for constructing a co-dimension 3 completely non-holonomic 
sub-bundle on the standard 7 sphere (or more general on a $4n+3$ dimensional standard sphere). 
In the latter case such a construction based on the Hopf bundle is well-known. Our method provides 
an alternated simple proof for the standard sphere $\mathbb{S}^7$. 
\end{abstract}

\maketitle

\tableofcontents
\thispagestyle{empty}


\section{Introduction}

In the area of differential geometry it is a natural problem to descide whether
a ``famous manifold'' has a particular geometric structure.
In this paper we deal with a so called sub-Riemannian
structure and explicitly contruct an example of co-dimension three on one of the exotic $7$ spheres. 

A sub-Riemannian structure on a manifold $M$ is defined through a sub-bundle $\mathcal{H}$ in the tangent bundle $T(M)$ equipped  with an inner product  and 
such that evaluations of vector fields taking values in $\mathcal{H}$ together with their iterated Lie brackets span the whole tangent space at each point. Together with 
these data $M$ is referred to as a sub-Riemannian manifold. 
A sub-bundle of the above type  is called {\it completely non-holonomic} or {\it bracket generating}  and  - 
in a sense - this notion is opposite to a foliation
structure. Corresponding to the {\it Frobenius theorem}  global
connectivity on connected components by (piecewise) horizontal curves is a basic geometric property of a 
sub-Riemannian structure (see \cite{Ch}). From an analytic point of view the bracket generating property implies sub-ellipticity of corresponding second order differential operators ({\it ''Sum of Squares''}) \cite{Ho}. 
On manifolds with this structure we can define an operator, called sub-Laplacian, which reflects various geometric properties (see for examples, 
\cite{St, Mo, Er1, Er2} and references therein). So one may say that the interest in the existence of a sub-Riemannian structure on a given manifold is caused by its various geometric and analytic implications.

Contact manifolds are among the most studied sub-Riemannian structures. Recall that a contact structure is
of co-dimension one and can be described in terms of a special kind of one form. In most of the concrete cases, 
the completely non-holonomic sub-bundle admits a natural inner product which is the restriction of a given Riemannian metric. 
Some examples of this type originate from the theory of dynamical systems (see, for example \cite{Ag}).

It was proved in \cite{SH} that every Brieskorn manifold has a contact structure and a generalization to a submanifold of a non-compact K\"ahler manifold
has been given in \cite{Va}. Here the one form for defining a contact structure is given as a restriction of a one-form $\theta$ whose differential $d\theta$ is the K\"ahler form. 

In particular, a $7$-dimensional Brieskorn manifold, which is an exotic sphere and 
realized as a submanifold in the complex vector space $\mathbb{C}^{10}$,
has a contact structure.  However, it is not known whether a $7$ dimensional exotic sphere has a higher co-dimensional sub-Riemannian structure. 
We mention that the standard $7$ sphere has such a structure of
co-dimension $3$ (there are several). This has been proved in
\cite{MM} (see \cite{St}, \cite{Mo}, and \cite{BF}).

In this paper we show that
the Gromoll-Meyer exotic sphere 
has a co-dimension $3$ sub-Riemannian structure.
Our description is valid also for  
the standard case 
constructed in \cite{MM}.

So, in $\S 2$ we define a candidate
of a completely non-holonomic sub-bundle 
on the base space of a principal bundle and explain
a method consisting of 
three steps for proving that it is completely non-holonomic.

Then in $\S 3$ we apply the method to 
the standard $7$ sphere $S^{7}$ and give a simple proof of 
the existence of a co-dimension $3$ completely non-holonomic sub-bundle recovering a result in \cite{MM}. 
In $\S 4$ and $\S 5$, we show our main theorem namely that the Gromoll-Meyer exotic $7$ sphere $\Sigma_{GM}^{7}$
has a co-dimension $3$ completely non-holonomic sub-bundle.

Parts of the arguments are based on methods in linear algebra. The non-trivial part consists in selecting four local vector fields
according to the point among the candidates which generate (together with the evaluation of their brackets) 
the whole tangent space at each point in $\Sigma^{7}_{GM}$. 

\section{Principal bundles and horizontal subspaces} 

We explain a standard procedure of defining a sub-bundle in the tangent
bundle on the base space of a principal bundle by considering a possible extension of the structure
group.  By this method we obtain our ''candidate'' of a completely non-holonomic sub-bundle on the base space of the principal bundle.

\subsection{First step : Candidate of a non-holonomic sub-bundle}

Let $\pi_{G}:P\to N$ be a principal bundle with the structure group $G$.
We denote the action of $G$ on $P$ by
\[
E:P\times G\to N\cong P/G,\quad (p,g)\mapsto E_{g}(p):=p\cdot g^{-1}
\]
and assume that the action is isometric with respect to a Riemannian metric
$$(\cdot\,,\,\cdot)_{p}:T_{p}(P)\times T_{p}(P)\to \mathbb{R}.$$
Let $K\subset G$ be a closed subgroup, then the restriction of the action $E$ to the subgroup $K$ gives a principal bundle
\[
\pi_{K}:P\to P/K=:M
\]
with the structure group $K$.

We have two orthogonal decomposition of the tangent bundle $T(P)$
\[
T(P)=V^{G}\oplus H^{G}=V^{K}\oplus H^{K},~V^{K}\subset V^{G},\quad
H^K\supset H^{G},
\]
where $V^{G}$ is the sub-bundle of $T(P)$ tangent to 
the action of $G$ and $H^{G}$ is its orthogonal complement.
$V^{K}$ and $H^{K}$ are defined in the same way.

\vspace{1ex}\par 
Let $\mathfrak{g}$ and $\mathfrak{k}$ be the Lie algebra of the group $G$ and $K$, respectively. Let $A\in\mathfrak{g}$ then 
we denote by $\tilde{A}$ the vector field on $P$ defined by 
\[
\tilde{A}(f)(p)=\tilde{A}_{p}(f):=\frac{d}{dt}\left(f(p\cdot e^{tA})\right)_{|t=0}.
\]
Recall that $\tilde{A}$ is called ``fundamental vector field'' corresponding to $A\in \mathfrak{g}$. It holds
\begin{align}
&(dE_{g})_{p}(\tilde{A}_{p})=\tilde{Ad_{g}(A)}_{E_{g}(p)},\label{Ad equivariant 1}\\
&(dE_{g})_{p}(H_{p}^{G})=H_{E_{g}(p)}^{G}~\text{for}~g\in G \label{Ad equivariant 2}
\intertext{and}
&(dE_{k})_{p}(H_{p}^{G})=H_{E_{k}(p)}^{G}~\text{for}~k\in K.\label{Ad equivariant 3}
\intertext{Also we have}
&(dE_{k})_{p}(H_{p}^{K})=H_{E_{k}(p)}^{K}~\text{for}~k\in K.\label{Ad equivariant 4}
\end{align}

Then by the relation
\[
(d\pi_{G})_{p}(H_{p}^{G})=(d\pi_{G})_{E_{g}(p)}\bigr((dE_{g})_{p}(H_{p}^{G})\bigr)
=(d\pi_{G})_{E_{g}(p)}(H_{E_{g}(p)}^{G})
\]
and the assumption that $E$ acts isometrically, we can define a Riemannian metric on
$T(N)$,
which is called a {``submersion metric''}. The splitting $T(P)=V^G\oplus H^G$ gives a connection on the principal bundle
$$\pi_{G}:P\to N\cong P/G.$$
A similar result holds for the principal bundle $\pi_{K}:P\to M\cong P/K$. Moreover, the relation
\[
(dE_{k})_{p}(H_{p}^{G})=H_{E_{k}(p)}^G
\]
allows us to descend $H^G$ 
to a sub-bundle in $T(P/K)$, which we denote by 
$\mathcal{H}$.
\vspace{1ex}\par
{\it Here  our aim is to show that the sub-bundle $\mathcal{H}$ is a completely non-holonomic sub-bundle on $T(P/K)$ for particular cases.}

\subsection{Second step : Bracket calculation}

Let $X$ and $Y$ be local vector fields around a point $q\in P/K\cong
M$ taking values in $\mathcal{H}$ and 
we denote by $\tilde{X}$ and $\tilde{Y}$ their horizontal lifts
according to the connection defined by the horizontal sub-bundle
$H^K$. Then we can assume that both lifts take values in $H^{G}$:
\[
d\pi_{K}(\tilde{X})=X,~d\pi_{K}(\tilde{Y})=Y.
\]
So instead of calculating the bracket $[X,\,Y]$,
we calculate $[\tilde{X},\,\tilde{Y}]$. Then 
we have
\[
d\pi_{K}\bigr([\tilde{X},\,\tilde{Y}]\bigr)=[d\pi_{K}(\tilde{X}),\,d\pi_{K}(\tilde{Y})]
=[X,\,Y].
\]

\subsection{Third step : Total space $P$ is a compact group}

We assume that the total space $P$ itself is a Lie group with the Lie
algebra
$\mathfrak{p}$ and is equipped with an invariant
metric
$(\cdot,\,\cdot)$
under the action of the structure group $G$.

We denote the left and right multiplication by
\[
L_{a}:P\to P,~L_{a}(x)=a\cdot x~\text{and}~R_{a}:P\to
P,~R_{a}(x)=x\cdot a,\hspace{2ex} (x,a\in P), 
\]
respectively. As ususal the tangent space $T(P)$ is identified with $P\times\mathfrak{p}$ through the map 
\[
\mathfrak{p}\ni u\mapsto \tilde{u}_{p}\in T_{p}(P),
\]
where $\tilde{u}$ is a left invariant vector field with the value $u$ at the identity, i.e. 
\begin{equation*}
\tilde{u}_{p}=(dL_{p})_{Id}(u).
\end{equation*}
\par 
Equivalently we may define $\tilde{u}$ in form of a derivation as:
\begin{equation*}
\tilde{u}_{p}(f)=\tilde{u}(f)(p)
=\frac{d}{dt}\left(f(p\cdot e^{tu})\right)_{|t=0}, \hspace{3ex} \mbox{where} \hspace{3ex} f\in C^{\infty}(P).
\end{equation*}

\begin{defn}\label{one form}
Let $A\in\mathfrak{g}$ and define a one form $\theta^{A}$
on $P$ by
\[
\theta^{A}_{p}:T_{p}(P)\to \mathbb{R},\quad \theta^{A}_{p}(U)
:=(\tilde{A}_{p},\,U)_{p},
\]
\end{defn}
where $\tilde{A}$ is the fundamental vector field corresponding to $A\in\mathfrak{g}$.
So, the subspaces $H^G_{p}$ and $H^{K}_{p}$ are characterized by
\begin{align*}
&H_{p}^{G}=\{~U\in T_{p}(P)~|~\theta^{A}_{p}(U)=0~\text{for}~{^\forall}A\in\mathfrak{g}~\},\\
&H_{p}^{K}=\{~U\in T_{p}(P)~|~\theta^{A}_{p}(U)=0~\text{for}~{^\forall}A\in\mathfrak{k}~\}.
\end{align*}
\par
Let $p\in P$ and put $\mathfrak{h}_{p}=(dL_{p^{-1}})_{p}(H^{G}_{p})$. In order to prove the complete non-holonomic property of the 
sub-bundle $\mathcal{H}$ we show: 
\begin{prop}\label{step 3}
For each $q\in P/K$, there is $p\in P$ such
that $\pi_{K}(p)=q$ and
\begin{equation}\label{bracket generating property 1}
\big{\{}\tilde{A}_{Id}~|~A\in\mathfrak{k}\big{\}}
+\mathfrak{h}_{p}+[\mathfrak{h}_{p},\,\mathfrak{h}_{p}]
=\mathfrak{p}.
\end{equation}
We show this property for the standard $7$ sphere case. In case to the  Gromoll-Meyer exotic sphere we prove
\begin{equation}\label{bracket generating property 2}
Ad_{p}\bigr(\,
\{\tilde{A}_{Id}~|~A\in\mathfrak{k}\}+\mathfrak{h}_{p}+[\mathfrak{h}_{p},\,\mathfrak{h}_{p}]\,
\bigr)=\mathfrak{p}.
\end{equation}
\end{prop}
\par 
Fix a point $p\in P$ and let $\tilde{X}^{i}, \:i=0,1,\cdots$ be 
horizontal vector fields defined locally around $p$ which take values in $H^G$ and form a basis of $H_{p}^{G}$. 
Also let $v_{i}\in\mathfrak{h}_{p}$ such that $(dL_{p})_{Id}(v_{i})=\tilde{X}^{i}_{p}$, 
then $\{v_i\}$ is a basis of $\mathfrak{h}_{p}$ and 
we show 
\begin{prop}\label{difference}
\[
[\tilde{X}^{i},\,\tilde{X}^{j}]_{p}\pm (dL_{p})_{Id}\big{(}[v_{i},v_{j}]\big{)}\in H^G_{p},
\] 
for particular pairs of horizontal vector fields $\tilde{X}^{i}$ 
or some type of sums of brackets of such vector fields
\[
\sum c_{ij}^{\ell}[\tilde{X}^{i},\,\tilde{X}^{j}]_{p}\pm(dL_{p})_{Id}
\left(\sum c_{ij}^{\ell}[v_{i},v_{j}]\right)\in H^G_{p}.
\]
The sign $\pm$ will be chosen according to the cases.

Finally, we show that the vectors $\{\tilde{X}^{i}_{p}\}_{i}$
and $\{\sum c_{ij}^{\ell}[\tilde{X}^{i},\tilde{X}^{j}]_{p}\}_{\ell}$
and evaluations 
~$\{\tilde{A}_{p}~|~A\in\mathfrak{g}\}$ of fundamental vector fields span the tangent space $T_{p}(P)$. 
\end{prop}

\section{A co-dimension $3$ sub-Riemannian structure on the standard $7$ sphere} 

We give a simple proof for the existence of a co-dimension $3$ 
completely non-holonomic sub-bundle on the standard $7$ sphere
(cf. \cite{MM} for a more direct approach). The
method is valid for all $4n+3$ dimensional standard spheres.

First, we describe  a co-dimension $3$ sub-bundle in the tangent bundle
of the
standard $7$ sphere $S^{7}$.

Let $\mathbb{H}$ be the quaternion number field:
\[
\mathbb{H}=\bigr\{h=h_0+h_{1}{\bf i}+h_{2}{\bf j}+h_{3}{\bf k}~\bigr|~h_{i}\in\mathbb{R}\bigr\},
\]
with the usual product and conjugation operations:
\begin{align*}
&{\bf i}^2={\bf j}^2={\bf k}^2=-1,~{\bf i}{\bf j}={\bf k}=-{\bf j}{\bf
  i},~{\bf j}{\bf k}=-{\bf k}{\bf j}={\bf i},~\text{etc.}\\
&\overline{h}=h_0-h_{1}{\bf i}-h_{2}{\bf j}-h_{3}{\bf k}~\,\,\text{and}\,\,~|h|=\sqrt{h\overline{h}}.
\end{align*}

Let {Sp(2)} be the group of quaternionic 
symplectic $2\times 2$ matrices:
\[
\text{Sp(2)}
=\biggr\{~p=\begin{pmatrix}x&y\\w&z\end{pmatrix}~\Bigr|~x,y,w,z\in\mathbb{H},~p\cdot
p^*=p^{*}\cdot p=Id~\biggr\},
\]
where
\[
p^{*}=\begin{pmatrix}\overline{x}&\overline{w}\\\overline{y}&\overline{z}\end{pmatrix}
\]
is the adjoint matrix of $p$. We denote its Lie algebra
by $\mathfrak{sp}(2)$:
\[
\mathfrak{sp}(2)=\left\{
\begin{pmatrix}\alpha&\beta\\-\overline{\beta}&\gamma\end{pmatrix}~\Bigr|~
\alpha=-\overline{\alpha},\gamma=-\overline{\gamma},\beta\in\mathbb{H}~
\right\}.
\]

Let $G$ = Sp(1)$\times$ Sp(1)
$=\{(\lambda,\,\mu)~|~|\lambda|=|\mu|=1~\}$
and write  $\mathfrak{g}$ for  its Lie algebra. 
We define the action 
of the group $G$ on {Sp(2)}
by the right multiplication:
\begin{align*}
&{R}:\text{Sp(2)}\times G\to \text{Sp(2)},\quad (p;\lambda,\,\mu)
\mapsto R_{(\lambda,\,\mu)}(p)\\
&R_{(\lambda,\,\mu)}(p)=
\begin{pmatrix}x\overline{\lambda}&y\overline{\mu}\\w\overline{\lambda}&z\overline{\mu}
\end{pmatrix} 
=\begin{pmatrix}x&y\\w&z\\\end{pmatrix}
\begin{pmatrix}\overline{\lambda}&0\\0&\overline{\mu}\\\end{pmatrix},\\  
&R_{(\lambda_{1},\,\mu_{1})}\circ R_{(\lambda,\,\mu)}=R_{(\lambda_1\cdot\lambda,\,\mu_{1}\cdot \mu)}
\end{align*}
and {we as well consider} 
its restriction to the subgroup 
$K=\text{Sp(1)}\times\{Id\}\subset G$. 
\bigskip

Then we have two principal bundles. One is {(cf. \cite{GM})}
 \begin{align*}
&\pi_{G}\,:\,\text{Sp(2)}\longrightarrow S^{4},
  ~p=\begin{pmatrix}x&y\\w&z\\\end{pmatrix}\mapsto \bigr(2y\cdot\overline{z},
  |y|^2-|z|^2\bigr)\in S^4,
\end{align*}
with the base space $S^4$ and the structure group $G\cong \text{Sp(1)}\times \text{Sp(1)}$,
and {the other} 
\begin{align*}
&\pi_{K}\,:\,\text{Sp(2)}\longrightarrow S^{7}, 
~p=\begin{pmatrix}x&y\\w&z\\\end{pmatrix}\mapsto \bigr(y, z\bigr)\in S^7.
\end{align*}
with the base space being the standard 7 sphere $S^7$ and with 
the structure group 
$$K= \text{Sp(1)} \times\{Id\}\subset G.$$ 

Since $K$ is a normal subgroup of $G$, we obtain a principal bundle called {\it Hopf bundle},
\[
\pi_{Hp}\,:\,S^{7}\longrightarrow S^{4},
\]
with the structure group {$G/K ~\cong~\text{Sp(1)}$}. We will denote the structure group action by $\overline{R}:S^7\times
G/K \to S^{7}$.

We identify (trivialize) the tangent bundle $T(\text{Sp(2)})$  
through the left invariant vector fields:
\begin{equation}\label{trivialization of tangent bundle}
\text{Sp(2)}\times \mathfrak{sp}(2)\cong T\text{(Sp(2))},
\end{equation}
where the identification is given by
\begin{equation*}
\mathfrak{sp}(2)\ni u\longmapsto \tilde{u}_{p}\in  T_{p}(\text{Sp(2)})
\end{equation*}
$\tilde{u}$~ is the left invariant vector field.
\vspace{1ex}\par 
Let $<\bullet,\,\bullet>$ be the inner product on $\mathfrak{sp}(2)$ given by
\[
<u,\,v>\,\stackrel{Def}=\,{\text{Re(Tr\,$u\cdot v^{*}$)}}=\text{Re}(x\cdot \overline{a}+
y\cdot \overline{b}+w\cdot \overline{c}+z\cdot \overline{d}),
\]
for $u=\begin{pmatrix}x&y\\w&z\end{pmatrix}$
and
$v=\begin{pmatrix}{a}&{b}\\{c}&{d}\end{pmatrix}$.

Since 
\[
<Ad_{g}(u),\,Ad_{g}(v)>=<u,\,v>,
\]
we can define a left and right invariant Riemannian metric 
$(\bullet,\,\bullet)$
on $\text{Sp(2)}$ 
through the above identification 
(\ref{trivialization of tangent bundle}).
 
Let $V^{G}_{p}\subset T_{p}(\text{Sp(2)})$ be the tangent space to the
fibers of the principal bundle $\pi_{G}: \text{Sp(2)} \to S^{4}$, 
that is, at $p \in \text{Sp(2)}$
\[
V_{p}^{G}=\left\{
X\in T_{p}(\text{Sp(2)})~\biggr|~ X=(dL_{p})_{Id}
\begin{pmatrix}x&0\\0&z\end{pmatrix},\,
x=-\overline{x},\,z=-\overline{z}\,\in\,\mathbb{H}
\right\}
\]
and let us denote hy $H^{G}$ the orthogonal complement to $V^{G}$ with respect to the
Riemannian metric $(\bullet,\,\bullet)$:
\begin{multline*}
H^{G}_{p}=\Big{\{}
Y\in
T_{p}(\text{Sp(2)})~\biggr|~
< (dL_{p^{-1}})_p(Y), u >=0, \\~\text{for
  any}~u=\begin{pmatrix}x&0\\0&z\end{pmatrix},
x=-\overline{x},\,z=-\overline{z}\in\mathbb{H}\Big{\}}.
\end{multline*}
So, $Y\in H_p^G$ is of the form
$$Y=(dL_{p})_{Id}\begin{pmatrix}0&a\\-\overline{a}&0\end{pmatrix}, \hspace{3ex} a \in \mathbb{H}. $$
\par 
In particular, at the identity element $Id=\begin{pmatrix}1&0\\0&1\end{pmatrix}\in \text{Sp(2)}$ and via the identification $T_{Id}(\text{Sp(2)})\cong \mathfrak{sp}(2)$ 
we have the orthogonal decomposition:
\[
\mathfrak{sp}(2)
\ni u=\begin{pmatrix}x&y\\-\overline{y}&z\end{pmatrix}=
\begin{pmatrix}x&0\\{0}&z\end{pmatrix}\oplus
\begin{pmatrix}0&y\\-\overline{y}&0\end{pmatrix}\,
\in \,
V^{G}_{Id}\oplus H^{G}_{Id}.
\]                                             

For $p\in \text{Sp(2)}$,
let $V_{p}^{K}$ be the tangent space 
to the orbit of the action $K$ through a
point $p\in \text{Sp(2)}$, then
\begin{align*}
&V_{p}^{K}=\left\{
X=(dL_{p})_{Id}(u)~\biggr|~u=\begin{pmatrix}x&0\\{0}&0\end{pmatrix}, 
~-\overline{x}=x\,\in\mathbb{H}~\right\}.\\
\intertext{
The orthogonal complement $H^{K}_{p}$ of $V^{K}_{p}$ is} 
&
H_{p}^K=\left\{
X=(dL_{p})(u)~\biggr|~u=\begin{pmatrix}0&y\\-\overline{y}&z\end{pmatrix},\,
-\overline{z}=z, y\,\in\mathbb{H}
\right\}.
\end{align*}
\par 
Since the decomposition $T(\text{Sp(2)}) \cong H^{G}\oplus V^{G}$ is $Ad$-equivariant (see (\ref{Ad equivariant 1}), (\ref{Ad equivariant
  2}) and (\ref{Ad equivariant 3})) under the orthogonal action of the group $G$ and the decomposition $T(\text{Sp(2)})\cong H^{K}\oplus V^{K}$
is $Ad$-equivariant 
under the orthogonal action of the group $K$,
each sub-bundle $H^{G}$ and $H^{K}$ defines 
a connection on the principal bundles
\begin{align*}
&\pi_{G}:\text{Sp(2)}\longrightarrow S^4,\\
&\pi_{K}:\text{Sp(2)}\longrightarrow S^{7},
\end{align*}
respectively.
Moreover,
since the sub-bundle $H^{G}$ is $Ad$-equivariant with respect
to the structure group $G$ it defines not only a sub-bundle in $H^{K}$ but 
also induces a sub-bundle $\mathcal{H}^{S}$ in $T(S^{7})$.
\vspace{1ex}\par 
The sub-bundle $H^{G}$ defines a connection of the Hopf bundle 
$\pi_{Hp}:S^7\to S^{4}\cong P^1(\mathbb{H})$,
that is the sub-bundle $d\pi_{K}(H^G)=\mathcal{H}^{S}$ satisfies
\[
\bigr(d\overline{R}_{\overline{g}}\bigr)_{x}\bigr(\mathcal{H}^{S}_{x}\bigr)=
\bigr(\mathcal{H}^{S}_{\overline{R}_{\overline{g}}(x)}\bigr),
~\overline{g}\in G/K,~x\in P/K.
\]

\begin{thm}{\em \cite{MM}}\label{MM}
The sub-bundle $\mathcal{H}^{S}$ is completely non-holonomic of step 2.
\end{thm}

\begin{proof}
We denote by $\Gamma(\mathcal{H}^{S})$ the space of vector fields taking
values in $\mathcal{H}^{S}$. 
Then
we show that the evaluations of 
vectors fields 
in 
$\Gamma(\mathcal{H}^{S})+[\Gamma(\mathcal{H}^{S}),\,\Gamma(\mathcal{H}^{S})]$
span the whole tangent space at each point in $S^7$.
\vspace{1ex}\par 
Fix a point $q\in S^{7}$ and let $X$, $Y$ be vector fields in $\Gamma(\mathcal{H}^{S})$ defined around $q$.  
We denote their horizontal lifts by $\tilde{X}$ and $\tilde{Y}$ defined around $p\in \text{Sp(2)}$ with $\pi_{K}(p)=q$. 
We may take vectors $u,v\in
H^{G}_{Id}\subset \mathfrak{sp}(2)$ such that
\begin{align*}
&\tilde{X}_{p}=(dL_{p})_{Id}(u),\tilde{Y}_{p}=(dL_{p})_{Id}(v),
\end{align*}
Since the fundamental vector field $\tilde{A}~(A\in\mathfrak{g})$ is left invariant on $\text{Sp(2)}$, we have
\begin{equation*}
\theta^{A}(\tilde{u})(p) \equiv <A,\,u> =~\text{constant}.
\end{equation*}
Hence
\[
\theta^{A}\big{(}[\tilde{X},\,\tilde{Y}]\big{)}(p)
=\theta^{A}\big{(}(dL_{p})_{Id}([u,v])\big{)}=<A,\,[u,v]>.
\]
In fact, this can be seen as follows:
\begin{align*}
&d\theta^{A}(\tilde{u}_{p},\,\tilde{v}_{p})=
d\theta^{A}(\tilde{X}_{p},\tilde{Y}_{p})\\
&=\tilde{X}_{p}(\theta^A(\tilde{Y}))
-\tilde{Y}_{p}(\theta^A(\tilde{X}))
-\theta^{A}([\tilde{X},\tilde{Y}])(p)\\
&=-\theta^{A}([\tilde{X},\tilde{Y}])(p),
\intertext{because by the definition of $\tilde{X}$ and $\tilde{Y}$,
$\theta^{A}(\tilde{X})=\theta^{A}(\tilde{Y})\equiv 0$.
On the other hand}
&d\theta^{A}(\tilde{u}_{p},\,\tilde{v}_{p})
=\tilde{u}_{p}(\theta^{A}(\tilde{v}))
-\tilde{v}_{p}(\theta^{A}(\tilde{u}))
-\theta^{A}([\tilde{u},\tilde{v}])(p)\\
&=-\theta^{A}([\tilde{u},\tilde{v}])(p)=-<A,[u,\,v]>.
\end{align*}
This means that
$$[\tilde{X},\,\tilde{Y}]_{p}-(dL_{p})_{Id}([u,\,v])\in H_{p}^{G}$$
(see Proposition \ref{difference}). Therefore
it is enough to 
show that
\[
H^{G}_{Id}+[H^{G}_{Id},\,H^{G}_{Id}]_{Id}=\mathfrak{sp}(2).
\]
For this purpose, we take a basis of $H^{G}_{Id}=\mathfrak{h}_{Id}$ 
\[
u_{0}=\begin{pmatrix}0&1\\-{1}&0\end{pmatrix},\,\,
u_{1}=\begin{pmatrix}0&{\bf i}\\{\bf i}&0\end{pmatrix},\,\,
u_{2}=\begin{pmatrix}0&{\bf j}\\{\bf j}&0\end{pmatrix},\,\,
u_{3}=\begin{pmatrix}0&{\bf k}\\{\bf k}&0\end{pmatrix}.
\]
Then
\begin{align*}
&[u_0,u_{1}]=\begin{pmatrix}2{\bf i}&0\\0&-2{\bf i}\end{pmatrix},\,
[u_0,u_{2}]=\begin{pmatrix}2{\bf j}&0\\0&-2{\bf j}\end{pmatrix},\,
[u_0,u_{3}]=\begin{pmatrix}2{\bf k}&0\\0&-2{\bf k}\end{pmatrix},\,\\
&[u_1,u_{2}]=\begin{pmatrix}2{\bf k}&0\\0&2{\bf k}\end{pmatrix},\,
[u_{1},u_{3}]=-\begin{pmatrix}2{\bf j}&0\\0&2{\bf j}\end{pmatrix},\,
[u_2,u_{3}]=\begin{pmatrix}2{\bf i}&0\\0&2{\bf i}\end{pmatrix}.
\end{align*}
Hence these $10$ vectors span the tangent space $\mathfrak{sp}(2)\cong
T_{Id}(\text{Sp}(2))$, which shows {Theorem} \ref{MM}
(see Propositions \ref{step 3} and \ref{difference}).
\end{proof}

\section{A co-dimension $3$ sub-Riemannian structure  
on the Gromoll-Meyer exotic $7$ sphere}
First we recall the definition of an exotic $7$ sphere 
(called Gromoll-Meyer exotic sphere) following the
description in \cite{GM}. We define a co-dimension $3$ sub-bundle 
in the tangent bundle of the Gromoll-Meyer exotic $7$ sphere. In the next section it will be shown 
that this sub-bundle is $2$ step completely non-holonomic. 

\vspace{1ex}\par 
Consider an action $E:\text{Sp(2)}\times G\to \text{Sp(2)}$ 
on $\text{Sp(2)}$ where $G= \text{Sp(1)}\times \text{Sp(1)}$,
\begin{align*}
&E:\text{Sp(2)}\times G\to \text{Sp(2)}\\
&E_{(\lambda,\,\mu)}(p)=
\begin{pmatrix}\lambda x\overline{\mu}&\lambda y\\\lambda w\overline{\mu}&\lambda z\end{pmatrix}
=\begin{pmatrix}\lambda &0\\0&\lambda \end{pmatrix}
\begin{pmatrix}x& y\\w&z\end{pmatrix}
\begin{pmatrix}\overline{\mu}&0\\0&1\end{pmatrix}\\
&E_{(\lambda_1,\,\mu_{1})}\,\circ\,E_{(\lambda,\,\mu)}=E_{(\lambda_1\cdot\lambda,\,\mu_{1}\cdot\mu)}
\end{align*}
and its restriction $E^{\Delta}$
to the subgroup 
$$\Delta=\{(\lambda,\lambda)\in \text{Sp(1)}\times \text{Sp}(1)\}\cong \text{Sp(1)}.$$
\par 
The total space $\text{Sp(2)}$ is equipped with the left and right invariant metric as in
$\S 3$.
Then
we have a principal bundle 
$P_{GM}=\{\text{Sp}(2),\Delta, \Sigma^{7}_{GM}\}$
with the orthogonal action of the 
structure group $\Delta$:
\[
\pi_{GM}:\text{Sp}(2)\longrightarrow \text{Sp(2)}/\Delta:=\Sigma^{7}_{GM},
\]
The base space is called {\it Gromoll-Meyer exotic $7$ sphere} (\cite{GM}). 
Also we have a principal bundle $P_{G}=\{\text{Sp}(2), G, S^4\}$  
with the orthogonal action of the structure group $G$
\[
\rho:\text{Sp(2)}\longrightarrow \text{Sp(2)}/G\cong S^4,
\]
where 
the identification of the base space with $S^{4}$ is induced through the map
\begin{align*}
&\text{Sp}(2)\longrightarrow S^{4}\subset \mathbb{H}\times \mathbb{R},\\
&p=\begin{pmatrix}x&y\\w&z\end{pmatrix}
\longmapsto
(2\overline{y}\cdot z,\,{|y|^2-|z|^2}).
\end{align*}
For  any $\lambda\in \mathfrak{sp}(1)$, we write 
$\lambda^{+}=\begin{pmatrix}\lambda&0\\0&0\end{pmatrix}$.

By definition of the action $E$, the fundamental vector field
$\tilde{A}$ for $A=(a,b)\in\mathfrak{sp}(1)\times \mathfrak{sp}(1)=\mathfrak{g}$
is given by
\begin{equation}\label{First_definition_Fundamental_Vectorfields}
\tilde{A}_{p}=(dR_{p})_{Id}(a\cdot Id)-(dL_{p})_{Id}(b^{+}).
\end{equation}
Hence for $u\in \mathfrak{sp}(2)$
\begin{align*}
\theta^{A}(\tilde{u})(p)
&=(\tilde{A}_{p},\tilde{u}_{p})_{p}
=\bigr((dR_{p})_{Id}(a\cdot Id)-(dL_{p})_{Id}(b^{+}),\tilde{u}_{p}\bigr)_{p}\\
&=\bigr((dR_{p})_{Id}(a)-(dL_{p})_{Id}(b^{+}),(dL_{p})_{Id}(u)\bigr)_{p}\\
&= <\,a\cdot Id,\,Ad_{p}(u)\,> - <b^{+},u>.\notag
\end{align*}

We consider two orthogonal decompositions of $T(\text{Sp(2)})$ as in $\S 3$ 
by vertical and horizontal sub-bundles according to the principal bundles
$P_{GM}$ and {$P_{G}$}:
\begin{align*}
T{(\text{Sp(2)})}&=V^{\Delta}\oplus H^{\Delta},\\
V^{\Delta}_{p}&=\Bigr\{~(dR_{p})_{Id}(\lambda\cdot{Id})\,-\,(dL_{p})_{Id}(\lambda^{+})
~\Bigr|~\overline{\lambda}=-\lambda~\Bigr\},\\
H^{\Delta}_{p}&=(V^{\Delta}_{p})^{\perp},\\
&\qquad\\
T{(\text{Sp(2)})}&=V^{G}\oplus H^{G},\\
V^{G}_{p}&=\Bigr\{(dR_{p})_{{Id}}\bigr(\lambda\cdot{Id}\bigr)
\,-\,(dL_{p})_{{Id}}(\mu^{+})~\Bigr
|~\lambda=-\overline{\lambda},\mu=-\overline{\mu}~\Bigr\},\\
H^{G}_{p}&=(V^{G}_{p})^{\perp}\\
&=(dL_{p})_{Id}\left(\Bigr\{~u\in\mathfrak{sp}(2)~\Bigr|~
u=\begin{pmatrix}0&\beta\\-\overline{\beta}&\gamma\end{pmatrix},~\text{and}
~\text{Tr}\,(Ad_{p}(u))=0\Bigr\}\right)\\
&=:(dL_{p})_{Id}(\mathfrak{h}_{p}).
\end{align*}
Then
\[
V^{\Delta}\subset V^{G} \hspace{3ex} \text{\it and} \hspace{3ex} H^{\Delta}\supset H^{G}
\]
and both of $H^{\Delta}$ and $H^{G}$ are ``{\it Ad-equivariant}'' with
respect to the action of the structure group
$\Delta$ and $G$, respectively. Hence they
define a connection on each principal bundle $P_{GM}$ and $P_{S}$. 
Since $H^{G}\subset H^{\Delta}$
and $H^{G}$ is Ad-equivariant under the structure group action of
$\Delta$, 
$H^{G}$ defines a sub-bundle $\mathcal{H}^{\Sigma}$ of
the tangent bundle $T(\Sigma^7_{GM})$.
\vspace{1ex}\par 
Now we state our main theorem.
\begin{thm}\label{main theorem}
${\mathcal{H}^{\Sigma}}$ is a co-dimension $3$ 
completely non-holonomic sub-bundle
in $T(\Sigma^7_{GM})$ and of step $2$.
\end{thm}
The proof of Theorem \ref{main theorem} will be given in the following section. 
\section{Proof of the main Theorem}
Let $q\in \Sigma^{7}_{GM}$ and
$p=\begin{pmatrix}x&y\\w&z\end{pmatrix}\in \pi_{GM}^{-1}(q)\subset
\text{Sp}(2)$. 
The group action of $\Delta$ allows us to choose $p$ with the
property that 
$xw^{-1}=:v, \;  (w\not=0)$ such that $v$ is of the form $v=v_{0}+v_{1}{\bf i}$. Then $p$ is rewritten as 
\begin{equation*}
p= 
\begin{pmatrix}
vw & y \\
w & -\overline{v}y
\end{pmatrix}\in \textup{Sp}(2) \hspace{2ex} \mbox{\it and } \hspace{2ex} |w|=|y|= \frac{1}{\sqrt{|v|^2+1}}.
\end{equation*}
Let ${\bf \rho} \in \{{\bf i}, {\bf j}, {\bf k}\}$ and take $A=(\rho, \rho) \in \mathfrak{sp}(1) \times \mathfrak{sp}(1)$. As before we denote by $\widetilde{A}$ the fundamental vector field of $A$. According to 
(\ref{First_definition_Fundamental_Vectorfields}) we obtain:  
\begin{align}
\ell_{\rho}:=
\textup{Ad}_p \Big{(} \widetilde{A}_{Id} \Big{)} 
=
\begin{pmatrix}
\rho-x \rho \overline{x} & -x \rho \overline{w} \\
-w \rho \overline{x} & \rho-w\rho \overline{w}
\end{pmatrix}. 
\label{defn_ell_rho}
\end{align}
We divide the proof into $4$ cases according to the possible values of $v$:
\begin{align*}
&\text{Case (I)}\,:x\not=0~\text{and}~w\not=0.~\text{We divide into {three} sub-cases},\\
&\qquad\qquad\text{Case (I-a)}: v_{1}\not=0~\text{and}~ v^2\not=-1,\\ 
&\qquad\qquad\text{Case (I-b)}: v^2=-1,\\
&\qquad\qquad\text{Case (I-r)}:\text{$v$ ~is real},\\
&\text{Case (II)}: x=0~\text{or}~w=0.
\end{align*}
\par 
The reason for dividing into these {$4$ cases} will be apparent from
formula (\ref{reason for classification}) below. 
The cases (I-r) and (II) can
be treated in a way similar to the {case of the standard 7 sphere.} We remark that case (I-a) is generic, that is the points in $\Sigma^{7}_{GM}$
having a fiber with such a property is open dense. So we give the proof for this case in full details.

\subsection{Lemmas from Gromoll-Meyer}

{We recall some basic {properties} from \cite{GM}. Let $q\in \Sigma^{7}_{GM}$ and let $X,Y$ be two 
local vector fields on $\Sigma^{7}_{GM}$ 
around a point $q$ taking values in
$\mathcal{H}^{\Sigma}$. 
By $\tilde{X}$
and $\tilde{Y}$ we denote their horizontal lifts to $P=\text{Sp}(2)$. Then we may assume that $\tilde{X},~\tilde{Y}\in \Gamma(H^{G})$, i.e.,
\begin{align*}
&\theta^{\lambda\cdot Id}(\tilde{X})=0,\hspace{2ex} \theta^{\mu^{+}}(\tilde{X})=0,\hspace{3ex}  \forall \; \lambda,
\mu\in\mathfrak{sp}(1),
\end{align*}
and the same for~$\tilde{Y}$. 

We fix a point $p\in$ Sp(2), and put $(dL_{p})_{{Id}}(u)=\tilde{X}_{p}$ and $(dL_{p})_{{Id}}(v)=\tilde{Y}_{p}$ {with $u,v \in \mathfrak{sp}(2)$} 
that is, ${\tilde{u}}_{p}=\tilde{X}_{p}$ and ${\tilde{v}}_{p}=\tilde{Y}_{p}$.  
Then
\begin{align*}
d\theta^{\lambda\cdot Id}(\tilde{X}_{p},\tilde{Y}_{p})
=
&=\tilde{X}_{p}(\theta^{\lambda\cdot Id}(\tilde{Y}))-
\tilde{Y}_{p}(\theta^{\lambda\cdot Id}(\tilde{X}))-
\theta^{\lambda\cdot Id}([\tilde{X},\tilde{Y}])(p)\\
&=-\theta^{\lambda\cdot Id}([\tilde{X},\tilde{Y}])(p),
\end{align*}
since $\theta^{\lambda\cdot Id}(\tilde{Y}))\equiv 0$ and
$\theta^{\lambda\cdot Id}(\tilde{X}))\equiv 0$.  
On the other hand
\begin{align*}
&d\theta^{\lambda\cdot Id}(\tilde{X}_{p},\tilde{Y}_{p})
= d\theta^{\lambda\cdot Id}({\tilde{u}}_{p},{\tilde{v}}_{p})\\
&= {\tilde{u}}_{p}(\theta^{\lambda\cdot Id}({\tilde{v}})) - 
\tilde{v}_{p}(\theta^{\lambda\cdot Id}({\tilde{u}})) - 
\theta^{\lambda\cdot Id}([{\tilde{u}},{\tilde{v}}])(p)\\
&= {\tilde{u}}_{p}(<\lambda\cdot Id,\,Ad_{*}(v)>)
-{\tilde{v}}_{p}(<\lambda\cdot Id,\,Ad_{*}(u)>) 
- \theta^{\lambda\cdot Id}(Ad_{p}([u,\,v]))\\
&= <\lambda\cdot Id,\,Ad_{p}([u,\,v])>
-<\lambda\cdot Id,\,Ad_{p}([v,\,u])>
-<\lambda\cdot Id,\,Ad_{p}([u,\,v])>\\
&= <\lambda\cdot Id,\,Ad_{p}([u,\,v])>.
\end{align*}
Hence, if we put $(dL_{p})_{Id}(Z)=[\tilde{X},\,\tilde{Y}]_{p}$, then
\begin{lem}{\em (\cite{GM})}\label{basic lemma GM 1}
For all $\lambda = - \overline{\lambda} \in \mathbb{H}$: 
\[
\theta^{\lambda\cdot Id}([\tilde{X},\,\tilde{Y}])(p)+
\theta^{\lambda\cdot Id}([{\tilde{u}},\,{\tilde{v}}])(p)
=<\lambda\cdot{Id},\, Ad_{p}(Z+[u,v])>=0. 
\]
That is, 
let $X$ and $Y$ be any local vector fields on $\Sigma^{7}_{GM}$ taking values in
$\mathcal{H}^{\Sigma}$ with horizontal lifts $\tilde{X}$ and $\tilde{Y}$ around 
$p\in \textup{Sp}(2)$. We may find $u$ and $v$ in $\mathfrak{sp}(2)$ such that
$$(dL_{p})_{Id}(u)=\tilde{u}_{p}=\tilde{X}_{p}$$
and $(dL_{p})_{Id}(v)=\tilde{v}_{p}=\tilde{Y}_{p}$.
Let $Z\in\mathfrak{sp}(2)$ such that
$(dL_{p})_{Id}(Z)=[\tilde{X},\,\tilde{Y}]_{p}$. Then above
calculations
says
\begin{equation}\label{basic relation 1}
\text{{\em Tr}}(Ad_{p}(Z+[u,\,v]))=0.
\end{equation}
\end{lem}

Likewise, with the same notations as in the previous lemma we have: 
\begin{lem}{\em (\cite{GM})}\label{basic lemma GM 2}
{For all $\lambda = - \overline{\lambda} \in \mathbb{H}$: }
\begin{equation}\label{basic relation 2}
\theta^{\lambda^{+}}([\tilde{X},\,\tilde{Y}])(p)-\theta^{\lambda^{+}}([{\tilde{u}},\,{\tilde{v}}])(p)
=<\lambda^{+},\,Z-[u,v]>=0.
\end{equation}
\end{lem}

\subsection{Adjoint action and  a characterization of the sub-bundle}

We fix a point $p=\begin{pmatrix}x&y\\w&z\end{pmatrix}\in$ Sp(2)
and put 
\[
\mathfrak{h}_{p}:=(dL_{p^{-1}})_{p}(H^{G}_{p}).
\]
Then
\begin{align*}
&\mathfrak{h}_{p}
=\left\{~u=\begin{pmatrix}0&\beta\\-\overline{\beta}&\gamma\end{pmatrix}~\Bigr|~
x\beta\overline{y}-y\overline{\beta}\overline{x}+
w\beta\overline{z}-z\overline{\beta}\overline{w}
+y{\gamma}\overline{y}+z{\gamma}\overline{z}=0~
\right\}.
\end{align*}
Conversely, let 
\[
u=\begin{pmatrix}a&b\\-\overline{b}&-a\end{pmatrix}\in\mathfrak{sp}(2)
\] 
and $Ad_{p^{-1}}(u)$ be of the form
$\begin{pmatrix}0&\beta\\-\overline{\beta}&\gamma\end{pmatrix}$
then
\begin{equation}\label{basic condition}
\overline{x}\,{a}x-\overline{w}\overline{b}x+\overline{x}bw-\overline{w}aw=0.
\end{equation}
Hence
\begin{equation}\label{equation_formula_for_Ad_p_h_p}
Ad_{p}(\mathfrak{h}_{p})=
\left\{\begin{pmatrix}a&b\\-\overline{b}&\hspace{-0.2cm}-a\end{pmatrix}~\Bigr|~
\overline{x}\,{a}x-\overline{w}\overline{b}x+\overline{x}bw-\overline{w}aw=0
\right\}.
\end{equation}
Let $p=\begin{pmatrix}x&y\\w&z\end{pmatrix}\in \text{Sp(2)}$.
We solve the equation (\ref{basic condition})
by considering the {above} two cases separately and fix a basis of  
the space $Ad_{p}(\mathfrak{h}_{p})$:
\vspace{1ex}\\
{\it  Case {\em (I)} :} Assume that $x\cdot w\not=0$. Equation (\ref{basic condition}) is rewritten as
\begin{align}\label{basic equation}
&\overline{v}\cdot a\cdot v~-~ a=\overline{b}\cdot v-\overline{v}\cdot b,
\end{align}
where $v=xw^{-1}$ and
we have assumed that $v=v_{0}+v_1{\bf i}\not=0$. We present the solutions of the equation (\ref{basic equation}).
\vspace{1ex}\par 
Put $\theta_{a}:=\overline{v}\cdot a\cdot v~-~ a$. Then
the solutions are given as follows:
\begin{itemize}
\item[\textup{($S_1$)}] Clealry, the pair $(a,b)=(0,v)$ is a solution.
\item[\textup{($S_2$)}] Define $\displaystyle{b_{a}=-\frac{v\theta_{a}}{2|v|^2}=\frac{v\cdot a-|v|^2a\cdot v}{2|v|^2}}$, then
\[
{\overline{b}_a\cdot v-\overline{v}\cdot b_a}
=\frac{\theta_{a}\cdot |v|^2}{2|v|^2}+
\frac{|v|^2\cdot\theta_{a}}{2|v|^2}=\theta_{a},  \hspace{4ex} {(a=-\overline{a} \in \mathbb{H}).}
\]
Hence the pair $(a,b)=(a,b_{a})$
is a solution for any $a=-\overline{a}\in\mathbb{H}$.
\end{itemize}
With the solutions in (\textup{$S_1$}) and (\textup{$S_2$}) we define a basis 
$u_{0},\, u_{\bf i},\, u_{\bf j},\, u_{\bf k}$  of $Ad_{p}\bigr(\mathfrak{h}_{p}\bigr)$ as follows: 
\begin{align}
u_{0}(v)=u_{0}&=\begin{pmatrix}0&v\\-\overline{v}&0\end{pmatrix},\label{base 0}\\
u_{\bf i}(v)=u_{\bf i}&=\begin{pmatrix}{\bf i}&b_{\bf i}\\-\overline{b_{\bf i}}&-{\bf i}\end{pmatrix}
=\begin{pmatrix}1&\frac{v(1-|v|^2)}{2|v|^2}\\
\frac{\overline{v}(1-|v|^2)}{2|v|^2}&-1\end{pmatrix}
\begin{pmatrix}{\bf i}&0\\0&{\bf i}\end{pmatrix},\label{base i}\\
u_{\bf j}(v)=u_{\bf j}&=\begin{pmatrix}{\bf j}&b_{\bf j}\\-\overline{b_{\bf j}}&-{\bf j}\end{pmatrix}
=\begin{pmatrix}1&\frac{v-|v|^2\overline{v}}{2|v|^2}\\
\frac{{v}-|v|^2\overline{v}}{2|v|^2}&-1\end{pmatrix}
\begin{pmatrix}{\bf j}&0\\0&{\bf j}\end{pmatrix}
:=S(v)\begin{pmatrix}{\bf j}&0\\0&{\bf j}\end{pmatrix},\label{base j}\\
u_{\bf k}(v)=u_{\bf k}&=\begin{pmatrix}{\bf k}&b_{\bf k}\\-\overline{b_{\bf
      k}}&-{\bf k}\end{pmatrix}
=S(v)\begin{pmatrix}{\bf k}&0\\0&{\bf k}\end{pmatrix}.\label{base k}
\end{align}
\begin{rem}\label{(1,1) component}
By the definition of the space $Ad_{p}(\mathfrak{h}_{p})$
these four matrices have the property that
$\text{Tr}(u_{0})=\text{Tr}(u_{\bf i})=\text{Tr}(u_{\bf j})=\text{Tr}(u_{\bf k})=0$,
and {the $(1,1)$ components} of $Ad_{p^{-1}}(u_{\rho})$ vanish for $\rho \in \{ 0, {\bf i}, {\bf j}, {\bf k} \}$. 
\end{rem}
Next we explicitly calculate the commutators
$[u_0,\,u_{\bf i}]$,
$[u_0,\,u_{\bf j}]$, $[u_0,\,u_{\bf k}]$
$[u_{\bf i},\,u_{\bf j}]$ and $[u_{\bf i},\,u_{\bf k}]$, 
respectively. 
\begin{align}\label{List_of_commutators}
&[u_0,\,u_{\bf i}]=\begin{pmatrix}1-|v|^2&-2v\\-2\overline{v}&|v|^2-1\end{pmatrix}
\begin{pmatrix}{\bf i}&0\\0&{\bf i}\end{pmatrix}\\
&[u_0,\,u_{\bf j}]=
\begin{pmatrix}\frac{v^2(1-\overline{v}^2)}{|v|^2}&-2v_{0}\\-2v_{0}&\overline{v}^2-1)\end{pmatrix}
\begin{pmatrix}{\bf j}&0\\0&{\bf j}\end{pmatrix}
=:M(v)\begin{pmatrix}{\bf j}&0\\0&{\bf j}\end{pmatrix},\notag\\
&[u_0,\,u_{\bf k}]=
\begin{pmatrix}\frac{v^2(1-\overline{v}^2)}{|v|^2}&-2v_{0}\\-2v_{0}&\overline{v}^2-1)\end{pmatrix}
\begin{pmatrix}{\bf k}&0\\0&{\bf k}\end{pmatrix}
=M(v)\begin{pmatrix}{\bf k}&0\\0&{\bf k}\end{pmatrix},\notag\\
&[u_{\bf i},\,u_{\bf j}]
=\begin{pmatrix}2+\frac{v^2(1-|v|^2)(1-\overline{v}^2)}{2|v|^4}&-\frac{v_{1}(1-|v|^2)}{|v|^2}{\bf
  i}\\
-\frac{v_{1}(1-|v|^2)}{|v|^2}{\bf i}&2+\frac{(1-|v|^2)(1-\overline{v}^2)}{2|v|^2}\end{pmatrix}
\begin{pmatrix}{\bf k}&0\\0&{\bf k}\end{pmatrix}
{=:} B(v)\begin{pmatrix}{\bf k}&0\\0&{\bf k}\end{pmatrix},\notag\\
&[u_{\bf i},\,u_{\bf k}]
=-\begin{pmatrix}2+\frac{v^2(1-|v|^2)(1-\overline{v}^2)}{2|v|^4}&-\frac{v_{1}(1-|v|^2)}{|v|^2}{\bf
  i}\\
-\frac{v_{1}(1-|v|^2)}{|v|^2}{\bf i}&2+\frac{(1-|v|^2)(1-\overline{v}^2)}{2|v|^2}\end{pmatrix}
\begin{pmatrix}{\bf j}&0\\0&{\bf j}\end{pmatrix}=-B(v)
\begin{pmatrix}{\bf j}&0\\0&{\bf j}\end{pmatrix}. \notag
\end{align}

\subsection{Various linear bases of $\mathfrak{sp}(2)$}

According to the $4$ cases explained above we choose a basis of  
$Ad_{p}(\mathfrak{h}_{p})+[Ad_{p}(\mathfrak{h}_{p}),\,Ad_{p}(\mathfrak{h}_{p})]$
with the property 
that 
$$\mathfrak{sp}(2)=Ad_{p}\bigr((dL_{p^{-1}})_{p}V^{\Delta}_{p}+\mathfrak{h}_{p}\bigr)
+[Ad_{p}(\mathfrak{h}_{p}),\,Ad_{p}(\mathfrak{h_{p}})].$$
\vspace{1ex}\\
\quad(I-a) {We assume $v_{1}\not =0$, $v^2\not=-1$}. Hence $|v|^2-v^2\not=0$.
{\allowdisplaybreaks 
{Put} 
\begin{align}
\alpha(v)
=&\frac{8|v|^{4}+(1-\overline{v}^2)(1-|v|^2)(v^2+|v|^2)
}{2|v|^2(1-\overline{v}^2)(|v|^2-v^2)}\notag\\
&=\frac{4\overline{v}}{(1-\overline{v}^2)(\overline{v}-v)}
+\frac{(1-|v|^2)(v+\overline{v})}{2|v|^2(\overline{v}-v)},\label{alpha}
\end{align}}
and define matrices $U_{\bf j}$ and $U_{\bf k}$ by
\begin{align}\label{U_k_U_k}
&U_{\bf j}(v)=U_{\bf j}=\alpha(v)\cdot [u_0(v),\,u_{\bf j}(v)]-[u_{\bf i}(v),\,u_{\bf k}(v)]
{=:}T(v)\begin{pmatrix}{\bf j}&0\\0&{\bf j}\end{pmatrix}
~\quad\text{and }\\
&U_{\bf k}(v)=U_{\bf k}=\alpha(v)\cdot [u_0(v),\,u_{\bf k}(v)]+[u_{\bf i}(v),\,u_{\bf j}(v)]
=T(v)\begin{pmatrix}{\bf k}&0\\0&{\bf k}\end{pmatrix}.\notag
\end{align}
The above definition of $\alpha(v)$ ensures that the traces of these two matrices vanish:
\begin{align}\label{trace zero 3}
\text{Tr}(U_{\bf j})=0,\hspace{2ex} \text{and} \hspace{2ex} \text{Tr}(U_{\bf k})=0.
\end{align}
\begin{prop}\label{base I-a}
Under the assumptions that {$xw\not=0$, $(xw^{-1})^2\not=-1$} and with our notation in (\ref{defn_ell_rho}) the $10$
matrices 
\begin{equation*}
\big{\{} u_{0},\,u_{\bf i},\,u_{\bf j},\,u_{\bf k},\,[u_{0}, u_{\bf i}],~U_{\bf j},~U_{\bf k}, \ell_{\bf i}, \ell_{\bf j},\ell_{\bf k} \big{\}}
\end{equation*}
{form}  a linear basis {\em(}over $\mathbb{R}${\em )} of 
the space $\mathfrak{sp}(2)$.
\end{prop}
\begin{proof}
Assume that
\begin{align*}
\lambda_{1}\ell_{\bf i}+\lambda_{2}\ell_{\bf j}
+\lambda_{3}\ell_{\bf k}
+c_{0}u_{0}+c_{1}u_{\bf i}+c_{2}u_{\bf j}+c_{3}{ u_{\bf k}}
+d_{1}[u_{0},\,u_{\bf i}]
+d_{2}U_{\bf j}
+d_{3}U_{\bf k}=0.
\end{align*}
We show that the coefficients $\lambda_j, c_j,d_j$ vanish.  According to (\ref{trace zero 3}) we know that 
\[
\textup{Tr}~U_{\bf j}=\textup{Tr} ~U_{\bf k}=0,
\]
and {by} the explicit expressions of the matrices 
$u_{0},\,u_{\bf i},\,u_{\bf j},\,u_{\bf k},\,[u_{0},\,u_{\bf i}]$
their traces vanish too. Therefore we have: 
\begin{align*}
&\text{Tr}(\lambda_{1}\ell_{\bf i}+
\lambda_{2}\ell_{\bf j}
+\lambda_{3}\ell_{\bf k}
+c_{0}u_{0}+c_{1}u_{\bf i}+c_{2}u_{\bf j}+c_{3}u_{\bf k}
+d_{1}[u_{0},\,u_{\bf i}]
+d_{2}U_{\bf j}
+d_{3}U_{\bf k})\\
&=
\text{Tr}(\lambda_{1}\ell_{\bf i}+
\lambda_{2}\ell_{\bf j}
+\lambda_{3}\ell_{\bf k})=0.
\end{align*}
Hence
\[
2(\lambda_{1}{\bf i}+\lambda_{2}{\bf j}+\lambda_{3}{\bf k})
=x\cdot(\lambda_{1}{\bf i}+\lambda_{2}{\bf j}+\lambda_{3}{\bf k})\cdot \overline{x}
+w\cdot(\lambda_{1}{\bf i}+\lambda_{2}{\bf j}+\lambda_{3}{\bf k})\cdot \overline{w}
\]
and consequently,
\[
2|\lambda|\leq |\lambda||x|^2+|\lambda| |w|^2=|\lambda|,
\]
which implies that $\lambda_{1}=\lambda_{2}=\lambda_{3}=0$.

Since $v=v_{0}+v_{1}{\bf i}$ and the constant
$\alpha(v)=\alpha_{0}(v)+\alpha_{1}(v){\bf i}$ is of the same form,
the equation
\[c_{0}u_{0}+c_{1}u_{\bf i}+c_{2}u_{\bf j}+c_{3}u_{\bf k}+d_{1}[u_{0},\,u_{\bf i}]
+d_{2}U_{\bf j}+d_{3}U_{\bf k}=0
\]
can be separated into a system of two 
\begin{align}
&c_{0}u_{0}+c_{1}u_{\bf i}+d_{1}[u_{0},\,u_{\bf i}]=0,\label{equation 2}\\
&c_{2}u_{\bf j}+c_{3}u_{\bf k}
+d_{2}U_{\bf j}+d_{3}U_{\bf k}=0.\label{equation 3}
\end{align}
The equation (\ref{equation 2}) is rewritten as
\begin{align*}
&c_{1}+d_{1}(1-|v|^2)=0,\\
&{-c_{0}v{\bf i}}+c_{1}\frac{v(1-|v|^2)}{2|v|^2}-2d_{1}v=0,\\
&{c_{0}\overline{v}{\bf i}}+c_{1}\frac{\overline{v}(1-|v|^2)}{2|v|^2}-2d_{1}\overline{v}=0.
\end{align*}
Hence we have
\[
c_{0}=0,~c_{1}=0~\text{and}~d_{1}=0.
\]
\par 
Equation (\ref{equation 3}) is equivalent to the system
 \begin{align*}
0=c_{2}S(v)\begin{pmatrix}{\bf j}&0\\0&{\bf j}\end{pmatrix}
+c_{3}S(v)\begin{pmatrix}{\bf k}&0\\0&{\bf k}\end{pmatrix}
+d_{2}T(v)\begin{pmatrix}{\bf j}&0\\0&{\bf j}\end{pmatrix}
+d_{3}T(v)\begin{pmatrix}{\bf k}&0\\0&{\bf k}\end{pmatrix}, 
\end{align*}
where 
\begin{align*}
&S(v)=\begin{pmatrix}s_{11}&s_{12}\\s_{21}&s_{22}\end{pmatrix}=
\begin{pmatrix}1&\frac{v-|v|^2\overline{v}}{2|v|^2}\\\frac{{v}-|v|^2\overline{v}}{2|v|^2}&-1\end{pmatrix},
~\text{and we put}~T(v)=\begin{pmatrix}t_{11}&t_{12}\\t_{21}&t_{22}\end{pmatrix}.
\end{align*}
Multiplication from the right by ${\bf j}\cdot  \textup{Id}$ implies that: 
\begin{equation*}
0=(c_{2}+c_{3}{\bf i})S(v)+(d_{2}+d_{3}{\bf i})T(v)=0. 
\end{equation*}
We need to show that the two complex matrices $S(v)$ and $T(v)$
are linearly independent over the complex numbers $\mathbb{C}$.
Here the components $t_{11}$ and $t_{12}$ are explicitly given as
\begin{align*}
t_{11}&=\alpha(v)\cdot \frac{v^2(1-\overline{v}^2)}{|v|^2}+2+
\frac{v^2(1-|v|^2)(1-\overline{v}^2)}{2|v|^4}=
\frac{(1+|v|^2)\cdot v\cdot (1+\overline{v}^2)}{|v|^2(\overline{v}-v)},\\
t_{12}&=-\alpha(v)\cdot(v+\overline{v}) {-} \frac{(1-|v|^2)({v-\overline{v}})}{2|v|^2}
={2} \frac{(1+|v|^2)(1+\overline{v}^2)}{(1-\overline{v}^2)(v-\overline{v})}.
\end{align*}
These formulas are obtained by using the expression (\ref{alpha}) of the
constant $\alpha(v)$.

If there is a constant $\delta=\delta_{0}+\delta_{1}{\bf i}$ such that
\[
\delta\cdot S(v)+T(v)=0,
\]
then $\delta=-t_{11}=t_{22}$. Now we prove that 
\[
-t_{11}\cdot s_{12}+t_{12}\not=0.
\]
In fact, a straightforward calculation shows
\begin{align}
-t_{11}\cdot s_{12}+ t_{12}
&= 
-\frac{(1+|v|^2)\cdot v
\cdot (1+\overline{v}^2)}{|v|^2(\overline{v}-v)}
\cdot\frac{v(1-\overline{v}^2)}{2|v|^2}+{2}
\frac{(1+|v|^2)(1+\overline{v}^2)}{(1-\overline{v}^2)(v-\overline{v})}\notag\\
&= \frac{(1+|v|^2)(1+\overline{v}^2)^3}{2(v-\overline{v})\overline{v}^2(1-\overline{v}^2)} \ne 0, 
\label{reason for classification}
\end{align}
since we assumed that $v-\overline{v}\not=0$ and $v^2\not=-1$. 
\end{proof}
\bigskip

\quad(I-b) In this case we assume $v^{2}=-1$. {That is, there is a point $p=\begin{pmatrix}x&y\\w&z\end{pmatrix}$ on the fiber
of $q\in\Sigma^{7}_{GM}$} such that $v=xw^{-1}=\pm{\bf i}$. 
Then, according to the group action the point 
$$p^{\prime}
=\begin{pmatrix}{\bf j}&0\\0&{\bf j}\end{pmatrix}
\cdot p\cdot \begin{pmatrix}-{\bf j}&0\\0&1\end{pmatrix}$$
also lies in $\pi_{GM}^{-1}(q)$. Hence we may assume that $v=xw^{-1}={\bf i}$. 
Then matrices
$p\in\text{Sp}(2)$
in the fiber $(\pi_{GM})^{-1}(q)$ {have the form} 
\begin{equation}\label{definition_p}
p=\begin{pmatrix}{\bf i}w&y\\w&{\bf i}y\end{pmatrix},
~{^\forall w},\,{^\forall {y}}~\text{with}~|w|^2=|y|^2=\frac{1}{2}.
\end{equation}
The space $Ad_{p}(\mathfrak{h}_{p})$ is spanned by the $4$
matrices 
\begin{align*}
&u_{0}=u_{0}({\bf i})
=\begin{pmatrix}0&{\bf i}\\{\bf i}&0\end{pmatrix},~
u_{\bf i}=u_{\bf i}({\bf i})=\begin{pmatrix}{\bf i}&0\\0&-{\bf i}\end{pmatrix},\\
&u_{\bf j}=u_{\bf j}({\bf i})=\begin{pmatrix}{\bf j}&{\bf k}\\{\bf k}&-{\bf j}\end{pmatrix}, ~
u_{\bf k}=u_{\bf k}({\bf i})=\begin{pmatrix}{\bf k}&-{\bf j}\\-{\bf j}&-{\bf k}\end{pmatrix}.
\end{align*}
Their commutators are given as
\begin{align*}
&[u_{0},\,u_{\bf i}]=2\begin{pmatrix}0&1\\-1&0\end{pmatrix},~
[u_{0},\,u_{\bf j}]=-2\begin{pmatrix}{\bf j}&0\\0&{\bf j}\end{pmatrix},~
[u_{0},\,u_{\bf k}]=-2\begin{pmatrix}{\bf k}&0\\0&{\bf k}\end{pmatrix},\\
&[u_{\bf i},\,u_{\bf j}]=2\begin{pmatrix}{\bf k}&0\\0&{\bf k}\end{pmatrix},~
[u_{\bf i},\,u_{\bf k}]=-2\begin{pmatrix}{\bf j}&0\\0&{\bf j}\end{pmatrix},~
[u_{\bf j},\,u_{\bf k}]=4\begin{pmatrix}{\bf i}&1\\-1&{\bf i}\end{pmatrix},
\end{align*}
and we have 
\begin{align}\label{Calculations_Adjoint}
&Ad_{p^{-1}}([u_{0},\,u_{\bf i}])
=4\begin{pmatrix}-\overline{w}{\bf i}w&0\\0 &\overline{y}{\bf i}y\end{pmatrix},\,\,\,\,
Ad_{p^{-1}}([u_{0},\,u_{\bf j}])
=4\begin{pmatrix}0&\overline{w}{\bf k}y\\
\overline{y}{\bf k}w&0\end{pmatrix},\\
&Ad_{p^{-1}}([u_{0},\,u_{\bf k}])
=-4\begin{pmatrix}0&\overline{w}{\bf j}y\\\overline{y}{\bf j}w&0\end{pmatrix},\,\,\,\,
Ad_{p^{-1}}([u_{\bf j},\,u_{\bf k}])
=16\begin{pmatrix}0&0\\0&\overline{y}{\bf i}y\end{pmatrix}.\notag
\end{align}
{Moreover, note that 
\begin{align*}
&\textup{Ad}_{p^{-1}} (u_0)= 2 \begin{pmatrix} 0 & \overline{w} {\bf i} y \\ \overline{y} {\bf i} w & 0 \end{pmatrix}, \hspace{2ex}
\textup{Ad}_{p^{-1}} (u_{\bf i})= 2 \begin{pmatrix} 0 & \overline{w} y \\ \overline{y} w & 0 \end{pmatrix} \\
&\textup{Ad}_{p^{-1}} (u_{\bf j})= 4 \begin{pmatrix} 0 & 0 \\ 0 & \overline{y} {\bf j}y \end{pmatrix}, \hspace{5ex}
\textup{Ad}_{p^{-1}} (u_{\bf k})= 4 \begin{pmatrix} 0 & 0 \\ 0 & \overline{y}{\bf k} y \end{pmatrix}
\end{align*}
Therefore we obtain for all $\lambda=- \overline{\lambda} \in \mathbb{H}$: }
\begin{align}\label{Relation_theta_lambda_plus}
&\theta^{\lambda^{+}}\bigr(Ad_{p^{-1}}\bigr(u_{0}\bigr)\bigr)=0,
~\theta^{\lambda^{+}}\bigr(Ad_{p^{-1}}\bigr(u_{\bf i}\bigr)\bigr)=0,\\
& \theta^{\lambda^{+}}\bigr(Ad_{p^{-1}}\bigr(u_{\bf j}\bigr)\bigr)=0,
~\theta^{\lambda^{+}}\bigr(Ad_{p^{-1}}\bigr(u_{\bf k}\bigr)\bigr)=0,\notag\\
&\theta^{\lambda^{+}}\bigr(Ad_{p^{-1}}\bigr([u_{0},\,u_{\bf j}]\bigr)\bigr)=0,
~\theta^{\lambda^{+}}\bigr(Ad_{p^{-1}}\bigr([u_{0},\,u_{\bf k}]\bigr)\bigr)=0,\notag \\
& \theta^{\lambda^{+}}\bigr(Ad_{p^{-1}}([u_{\bf j},\,u_{\bf k}])\bigr)=0.\notag
\end{align}
That is the $(1,1)$-component of each among these $7$ matrices is zero. 
Now, put $w=a_0+a_{1}{\bf i}+b_{0}{\bf j}+b_{1}{\bf k}=a+b{\bf j}=-{\bf i}x\in \mathbb{H}$ and consider $\ell_{\rho}$ in (\ref{defn_ell_rho}),
\begin{prop}\label{base I-b}
Let $w=a_0+a_{1}{\bf i}+b_{0}{\bf j}+b_{1}{\bf k}=a+b{\bf j}=-{\bf i}x\in \mathbb{H}$.
In this case $|w|^2=|x|^2=|a|^2+|b|^2=\frac{1}{2}.$ 
Then it follows with the above notation: 
\begin{itemize}
\item[\textup{(i)}] 
If $|a|^2-|b|^2\not=\frac{1}{4}$, then the $10$ matrices 
\begin{align*}
&\ell_{\bf i}
=\begin{pmatrix}{\bf i}+{\bf i}{w}{\bf i}\overline{w}{\bf i}&-{\bf i}{w}{\bf i}\overline{w}
\\{w}{\bf i}\overline{w}{\bf i}&{\bf i}-{w}{\bf i}\overline{w}\end{pmatrix}, \hspace{1ex} 
\ell_{\bf j}
=\begin{pmatrix}{\bf j}+{\bf i}{w}{\bf j}\overline{w}{\bf i}&-{\bf i}{w}{\bf j}\overline{w}
\\{w}{\bf j}\overline{w}{\bf i}&{\bf j}-{w}{\bf j}\overline{w}\end{pmatrix},\\
&\ell_{\bf k}
=\begin{pmatrix}{\bf k}+{\bf i}{w}{\bf k}\overline{w}{\bf i}&-{\bf i}{w}{\bf k}\overline{w}
\\{w}{\bf k}\overline{w}{\bf i}&{\bf k}-{w}{\bf k}\overline{w}\end{pmatrix},\\
&u_{0}=\begin{pmatrix}0&{\bf i}\\{\bf i}&0\end{pmatrix},\,\,
~u_{\bf i}=\begin{pmatrix}{\bf i}&0\\0&-{\bf i}\end{pmatrix},~
~u_{\bf j}=\begin{pmatrix}{\bf j}&{\bf k}\\{\bf k}&-{\bf j}\end{pmatrix}
,~~u_{\bf k}=\begin{pmatrix}{\bf k}&-{\bf j}\\-{\bf j}&-{\bf k}\end{pmatrix},\\
&F_{{\bf i}}:=\begin{pmatrix}0&1\\-1&0\end{pmatrix}=\frac{1}{2}[u_{0},\,u_{\bf i}],
~F_{\bf j}:=\begin{pmatrix}{\bf j}&0\\0&{\bf j}\end{pmatrix}=-\frac{1}{2}[u_{0},\,u_{\bf j}],\\
&F_{\bf k}:=\begin{pmatrix}{\bf k}&0\\0&{\bf k}\end{pmatrix}=-\frac{1}{2}[u_{0},\,u_{\bf k}]
\end{align*}
are linearly independent over $\mathbb{R}$ and span $\mathfrak{sp}(2)$.

\item[\textup{(ii)}] 
If $|a|^2-|b|^2=\frac{1}{4}$, then the $10$ matrices $\ell_{\bf i}$, $\ell_{\bf j}$, $\ell_{\bf k}$, $u_0$, $u_{\bf i}$, 
$u_{\bf j}$, $u_{\bf k}$, $F_{\bf j}$, $F_{\bf k}$ and 
\begin{equation*}
F^{\prime}_{\bf i}:=\begin{pmatrix}{\bf i}&1\\-1&{\bf i}\end{pmatrix}=\frac{1}{4}[u_{\bf j},\,u_{\bf k}]
\end{equation*}
are linearly independent over $\mathbb{R}$ and span $\mathfrak{sp}(2)$.
\end{itemize}
\end{prop}

\begin{proof}

\text{(i):} Let $\mu_{i}, c_{i}, d_{i}~\in\mathbb{R}$ and assume
\begin{align}
&U=\mu_{1}\ell_{\bf i}+\mu_{2}\ell_{\bf j}+\mu_{3}\ell_{\bf k}+c_{0}u_{0}
+c_{1}u_{\bf i}+c_{2}u_{\bf j}+c_{3}u_{\bf k}\label{I-b null condition (i)}\\
&\qquad\qquad\qquad\qquad\qquad\qquad\qquad\qquad
+d_{1}F_{\bf i}
+d_{2}F_{\bf j}
+d_{3}F_{\bf k}=0.\notag
\end{align}
{Then we conclude that 
$\theta^{\lambda^{+}}(\,Ad_{p^{-1}}(U)\,)=0$ for any $\lambda=-\overline{\lambda}\in \mathbb{H}$. If we put $\mu =\mu_{1}{\bf i}+\mu_{2}{\bf j}+\mu_{3}{\bf k}$, 
then the identities (\ref{Relation_theta_lambda_plus}) 
together with 
\begin{equation*}
\textup{Ad}_{p^{-1}}(F_{\bf i})=2 \begin{pmatrix} -\overline{w} {\bf i} w & 0 \\ 0 & \overline{y} {\bf i} y \end{pmatrix} 
\end{equation*}
imply that }
\begin{equation}\label{(1,1) component =0}
\overline{x}\mu x+ \overline{w}\mu w-\mu
+d_{1}(-\overline{w}x+\overline{x}w)=-\overline{w}{\bf i}\mu {\bf i}w+\overline{w}\mu w -\mu
-2d_{1}\overline{w}{\bf i}w=0.
\end{equation}
Hence
\begin{equation*}
w\mu\overline{w}=\frac{\mu}{4}-\frac{{\bf i}\mu{\bf i}}{4}-{\frac{d_{1}}{2}}{\bf i}.
\end{equation*}
\par 
From this we have
\begin{align}
&x\mu \overline{x}=-{\bf i}w\mu \overline{w}{\bf i}
=-{\bf i}\left(\frac{\mu}{4}-\frac{{\bf i}\mu{\bf i}}{4}-{\frac{d_{1}}{2}}{\bf i}\right){\bf i}
=\frac{\mu}{4}-\frac{{\bf i}\mu{\bf i}}{4}-{{red} \frac{d_{1}}{2}} {\bf i}
=w\mu\overline{w}\label{identity I-b}.
\end{align}
The equation (\ref{I-b null condition (i)})
is rewritten as
\begin{align}
&\mu-{x}\mu\overline{x}+c_{1}{\bf i}+c_{2}{\bf j}+c_{3}{\bf
  k}+d_{2}{\bf j}+d_{3}{\bf k}=0,\label{equation I-b1}\\
&-x\mu\overline{w}+c_{0}{\bf i}+c_{2}{\bf k}-c_{3}{\bf j}+d_{1}=0,\label{equation I-b2}\\
&\mu -{w}\mu \overline{w}-c_{1}{\bf i}-c_{2}{\bf j}-c_{3}{\bf k}+d_{2}{\bf j}+d_{3}{\bf k}=0.\label{equation I-b3}
\end{align}
Using the identity (\ref{identity I-b}) we obtain from (\ref{equation I-b1}) and (\ref{equation I-b3}): 
\[
c_{1}=c_{2}=c_{3}=0.
\]
Then the  equation    (\ref{equation I-b2}) becomes
\[
w\mu\overline{w}-c_{0}+d_{1}{\bf i}=0,
\]
which implies $c_{0}=0$ and 
\begin{equation}\label{d{1}relation}
w\mu\overline{w}=-d_{1}{\bf i}.
\end{equation}
Then the (\ref{equation I-b3}) 
\[
\mu-w\mu\overline{w}+d_{2}{\bf j}+d_{3}{\bf k}
=\mu+d_{1}{\bf i}+d_{2}{\bf j}+d_{3}{\bf k}=0
\]
gives $\mu_{i}=-d_{i}$ for $i=1,2,3$ and (\ref{d{1}relation}) shows that

\[
w\mu \overline{w}=\mu_{1}{\bf i}
\]
This implies 
\[
\mu_{1}(|a|^2-|b|^2)=\frac{\mu_{1}}{4}.
\]
Hence $\mu_{1}=0$ and the identity $w \mu \overline{w}=0$ implies that $\mu_{2}=\mu_{3}=0$, which shows our assertion.
\vspace{1ex}\\ 
\textup{(ii)}:  As in the case {\text{(i)}} we assume that 
\begin{equation}
\mu_{1}\ell_{\bf i}+\mu_{2}\ell_{\bf j}+\mu_{3}\ell_{\bf k}+c_{0}u_{0}
+c_{1}u_{\bf i}+c_{2}u_{\bf j}+
\label{I-b null condition (ii)}
+c_{3}u_{\bf k}
+d_{1}F^{\prime}_{\bf i}
+d_{2}F_{\bf j}
+d_{3}F_{\bf k}=0.
\end{equation}
{By using the same arguments as before} this equation is separated into a system of three equations:
\begin{align}
&\mu-{x}\mu\overline{x}+c_{1}{\bf i}+c_{2}{\bf j}+c_{3}{\bf
  k}+d_{1}{\bf i}+d_{2}{\bf j}+d_{3}{\bf k}=0,\label{equation I-b1 2}\\
&-x\mu\overline{w}+c_{0}{\bf i}+c_{2}{\bf k}-c_{3}{\bf
  j}+d_{1}=0,\label{equation I-b2 2}\\
&\mu -{w}\mu \overline{w}-c_{1}{\bf i}-c_{2}{\bf j}-c_{3}{\bf k}
+d_{1}{\bf i}+d_{2}{\bf j}+d_{3}{\bf k}=0.\label{equation I-b3 2}
\end{align}
Instead of the relation (\ref{(1,1) component =0}),
we have
\[
\overline{x}\mu x+ \overline{w}\mu w-\mu=
-\overline{w}{\bf i}\mu {\bf i}w+\overline{w}\mu w -\mu=0
\]
since in this case the $(1,1)$ component of $Ad_{p^{-1}}(F^{\prime}_{\bf i})$ vanishes, cf. (\ref{Calculations_Adjoint}).
However, one still has the equality (\ref{identity I-b}): 
\[
x\mu \overline{x}=w\mu\overline{w}.
\]
Hence $c_{1}=c_{2}=c_{3}=0$ and from (\ref{equation I-b2 2}) we see that $c_{0}=0$. Now (\ref{equation I-b3 2}) gives 
\[
\mu -{w}\mu \overline{w}+d_{1}{\bf i}+d_{2}{\bf j}+d_{3}{\bf k}
=\mu +2d_{1}{\bf i}+d_{2}{\bf j}+d_{3}{\bf k}=0
\]
and therefore
\[
\mu_{1}=-2d_{1}.
\]
\par 
Equation (\ref{equation I-b2 2}) 
\[
w\mu \overline{w}=\frac{1}{2}\mu_{1}{\bf i}
\]
implies the equality
\[
2\mu_{1}=\mu_{1}(|a|^2-|b|^2)=\frac{\mu_{1}}{4}.
\]
Hence $\mu_{1}=0$, which also implies $\mu=0$,
since ${|\mu|}=|\mu_{1}|$. We have proved {\text{(ii)}}.
\end{proof}
\noindent
{\it Case  {\em (I-r):}}  If $v_{1}=0$, i.e., when {$v\ne 0$} is real, instead of $U_{\bf j}$ and
$U_{\bf k}$, we choose the matrices
$[u_{0},\,u_{\bf j}]$ and $[u_{0}, \, u_{\bf k}]$. 
\begin{prop}\label{base:v is real}
With the notation in (\ref{defn_ell_rho}) the $10$ matrices 
\begin{align*}
{\Bigl{\{}\ell_{\bf i}, \ell_{\bf j}, \ell_{\bf k}, u_{0},\,u_{\bf i},\,u_{\bf j},\,u_{\bf k},\,[u_{0}, u_{\bf i}],~[u_{0},\,u_{\bf j}], [u_{0}, \,u_{\bf k}]\Bigr\}}
\end{align*}
form a basis of $\mathfrak{sp}(2)$.
\end{prop}
\begin{proof}
{ 
From the explicit list of commutators (\ref{List_of_commutators}) we observe that  
\begin{equation*}
\textup{Tr}\big{(}[u_0, u_{\rho}] \big{)}=\textup{Tr}\big{(} u_{\rho}\big{)} =0, \hspace{4ex} \rho \in \{ 0,{\bf i}, {\bf j}, {\bf k} \big{\}}. 
\end{equation*}}
So the linear independence of the systems: 
\begin{equation*}
\mathcal{S}_1:=\Big{\{} \ell_{\bf i},\ell_{\bf j},\ell_{\bf k}\Big{\}} \hspace{3ex} \mbox{and} \hspace{3ex}  
\mathcal{S}_2:=\Big{\{}u_{0},\,u_{\bf i},\,u_{\bf j},\,u_{\bf k},\,[u_{0},~u_{\bf i}],\,[u_{0},~u_{\bf j}],\,[u_{0},~u_{\bf k}]\Big{\}}
\end{equation*}
is proved in the same way as in case (I-a). Linear independence of the seven matrices in $\mathcal{S}_2$ follows in the same way as the standard 
case in $\S 3$.
\end{proof}
\noindent
{\it Case {\em(II):}} Now we assume that $x=0$ or $w=0$. {Then (\ref{equation_formula_for_Ad_p_h_p}) shows that}
\begin{equation}\label{Equation_Ad_case_II}
Ad_{p}(\mathfrak{h}_{p})=\left\{\begin{pmatrix}0&b\\-\overline{b}&0\end{pmatrix}~\Bigr|~
b\in\mathbb{H}
\right\}.
\end{equation}
Choose the basis 
$u_{0}=\begin{pmatrix}0&1\\-1&0\end{pmatrix}$,
$u_{\bf i}=\begin{pmatrix}0&{\bf i}\\{\bf i}&0\end{pmatrix}$,
$u_{\bf j}=\begin{pmatrix}0&{\bf j}\\{\bf j}&0\end{pmatrix}$,
$u_{\bf k}=\begin{pmatrix}0&{\bf k}\\{\bf k}&0\end{pmatrix}$ 
of the space (\ref{Equation_Ad_case_II}). 
\begin{prop}\label{base II}
With the notation in (\ref{defn_ell_rho}) the $10$ matrices 
\[
\Big{\{} \ell_{\bf i},~\ell_{\bf j},~\ell_{\bf k},~u_{0},~u_{\bf i},~u_{\bf j},
~u_{\bf k},~[u_{0},~u_{\bf i}],~[u_{0},~u_{\bf j}],
~[u_{0},~u_{\bf k}]\Big{\}}
\]
form a basis of $\mathfrak{sp}(2)$. 
\end{prop}
\bigskip

\subsection{Final part of the proof }

We complete the proof of the main Theorem \ref{main theorem}
for the cases (I-a) and (I-b) based on Proposition 
\ref{base I-a} and Proposition \ref{base I-b}, respectively. 
{The remaining cases are proved in the same way via Propositions \ref{base:v is real}, \ref{base II}, 
 and \ref{difference}.}
\vspace{1ex}\\
{\it Case {\em (I-a)}:}\quad Let $q\in \Sigma^{7}_{GM}$ and take a {point 
$p=\begin{pmatrix}x&y\\w&z\end{pmatrix}\in \text{Sp(2)}$
in the fiber} such that 
\begin{equation*}
x\cdot w^{-1}=v_{0}+v_{1}{\bf i}:=v \hspace{3ex} \mbox{and} \hspace{3ex} v_1 \ne 0, \hspace{1ex} {v^2 \ne -1}. 
\end{equation*}
\par
{
Locally around $p$ we define vector fields $\tilde{X}^{\rho}$, 
$\rho\in \{0,{\bf i},{\bf j},{\bf k}\}$ horizontal to the fibration 
$\text{Sp(2)}\to \Sigma^{7}_{GM}$ 
and taking values in ${H}^{G}$  such that 
\begin{equation*}
\big{(}dL_{p}\big{)}_{Id}(Ad_{p^{-1}}(u_{\rho}))=\tilde{X}^{\rho}_{p} \hspace{3ex} \mbox{for} \hspace{3ex} \rho\in \big{\{}0,{\bf i},{\bf j},{\bf k}\big{\}}, 
\end{equation*}}
where $u_{0},~u_{\bf i},~u_{\bf j},~u_{\bf k}$
are the matrices in $Ad_{p}(\mathfrak{h}_{p})$ 
defined in (\ref{base 0}) - (\ref{base k}).
\begin{prop}\label{final prop I-a}
The local vector fields
\[
{X^{\rho}:=(d\pi_{GM})(\tilde{X}^{\rho}) , \hspace{3ex} \rho \in \big{\{}0,{\bf i},{\bf j},{\bf k}\big{\}},}
\]
take values in $\mathcal{H}^{\Sigma}$ and the seven tangent vectors at $q=\pi_{GM}(p)$
\begin{align*}
&X^{0}_{q},\,X^{\bf i}_{q},\,X^{\bf j}_{q},\,X^{\bf k}_{q},\\
&[(d\pi_{GM})(\tilde{X}^{0}),\,(d\pi_{GM})(\tilde{X}^{\bf i})]_{q}
=[X^{0},\,X^{\bf i}]_{q},\\
&[\alpha(v)(d\pi_{GM})_{p}(\tilde{X}^{0}),\,(d\pi_{GM})_{p}(\tilde{X}^{\bf j})]_{q}
-[(d\pi_{GM})_{p}(\tilde{X}^{\bf i}),\,(d\pi_{GM})_{p}(\tilde{X}^{\bf k})]_{q}\\
&=\alpha(v)[X^{0},\,X^{\bf j}]_{q}-[X^{\bf i},\,X^{\bf k}]_{q}\\
&[\alpha(v)(d\pi_{GM})_{p}(\tilde{X}^{0}),\,(d\pi_{GM})_{p}(\tilde{X}^{\bf k})]_{q}
+[(d\pi_{GM})_{p}(\tilde{X}^{\bf i}),\,(d\pi_{GM})_{p}(\tilde{X}^{\bf j})]_{q}\\
&=\alpha(v)[X^{0},\,X^{\bf k}]_{q}+[X^{1},\,X^{\bf j}]_{q}
\end{align*}
span the tangent space $T_{q}(\Sigma^{7}_{GM})$.
\end{prop}
\begin{proof}
We define $Z_{i}\in \mathfrak{sp}(2)$ ($i=1,2,3$) by 
\begin{align*}
&(dL_{p})_{Id}\bigr(Z_{1}\bigr)=[\tilde{X}^{0},\,\tilde{X}^{\bf i}]_{p},\\
&(dL_{p})_{Id}\bigr(Z_{2}\bigr)=
\alpha(v)[\tilde{X}^{0},\,\tilde{X}^{\bf j}]_{p}-[\tilde{X}^{\bf i},\,\tilde{X}^{\bf k}]_{p},\\
&(dL_{p})_{Id}\bigr(Z_{3}\bigr)=
\alpha(v)[\tilde{X}^{0},\,\tilde{X}^{\bf k}]_{p}+[\tilde{X}^{\bf i},\,\tilde{X}^{\bf j}]_{p},
\end{align*}
where $\alpha(v)\in \mathbb{H}$ is given in (\ref{alpha}). 
Then, by Lemma \ref{basic lemma GM 2} 
we have:
\begin{align*}
0=&\text{Tr}\bigr(Ad_{p}(Z_{1})+[u_{0},u_{\bf i}])=
\text{Tr}\bigr(Ad_{p}(Z_{1}))=
\text{Tr}\bigr(Ad_{p}(Z_{1})-[u_{0},u_{\bf i}]),\\
0=&\text{Tr}\bigr(Ad_{p}(Z_{2})
+\alpha(v)[u_{0},u_{\bf j}]-[u_{\bf i},\,u_{\bf k}]\bigr)=\text{Tr}\bigr(Ad_{p}(Z_{2})\bigr)\\
=&\text{Tr}\bigr(Ad_{p}(Z_{2})-
\alpha(v)[u_{0},u_{\bf j}]+[u_{\bf i},\,u_{\bf k}]\bigr),\\
0=&\text{Tr}\bigr(Ad_{p}(Z_{3})
+\alpha(v)[u_{0},u_{\bf k}]+[u_{\bf i},\,u_{\bf j}]\bigr)=\text{Tr}\bigr(Ad_{p}(Z_{3})\bigr)\\
=&\text{Tr}\bigr(Ad_{p}(Z_{3})-
\alpha(v)[u_{0},u_{\bf k}]-[u_{\bf i},\,u_{\bf j}]\bigr).
\end{align*}
Also Lemma \ref{basic lemma GM 1} implies that 
\begin{align*}
&<{\lambda^{+}},\,Z_{1}-Ad_{p^{-1}}\bigr([u_{0},u_{\bf i}]\bigr)>=0,\\
&<{\lambda^+},\, Z_{2}
-Ad_{p^{-1}}\bigr(\alpha(v)[u_{0},u_{\bf j}]-[u_{\bf i},\,u_{\bf k}]\bigr)>=0,\\
&<{\lambda^+},\, Z_{3}-Ad_{p^{-1}}\bigr(\alpha(v)[u_{0},u_{\bf k}]+[u_{\bf i},\,u_{\bf j}]\bigr)>=0,
\end{align*}
for any $\lambda=-\overline{\lambda}\in\mathbb{H}$.
{Therefore, the vectors} 
\begin{align*}
&Ad_{p}(Z_{1})-[u_{0},u_{\bf i}], ~
Ad_{p}(Z_{2})-\alpha(v)[u_{0},u_{\bf j}]+[u_{\bf i},\,u_{\bf k}],\\
&Ad_{p}(Z_{3})-\alpha(v)[u_{0},u_{\bf k}]-[u_{\bf i},\,u_{\bf j}]
\end{align*}
belong to $Ad_{p}\bigr(\mathfrak{h}_{p}\bigr)$. Note that in Proposition \ref{base I-a}
we may replace $[u_{0},\,u_{\bf i}]$, $U_{\bf j}$ and $U_{\bf k}$ by $Ad_{p}(Z_{1})$, $Ad_{p}(Z_{2})$ and $Ad_{p}(Z_3)$, respectively. 
This proves Case (I-a).
\end{proof}
\noindent 
{\it Case {\em (I-b)}:} We treat case (i) of Proposition \ref{base I-b} and remark that  (ii) is proved in the same way as the case (I-a). 
Let $q\in \Sigma^{7}_{GM}$ and fix a point 
$$p=\begin{pmatrix}x&y\\w&z\end{pmatrix}\in \text{Sp}(2)$$
in the fiber} such that $v=x\cdot w^{-1}={\bf i}$.
Locally around $p$ we define vector fields $\tilde{Y}^{\rho}$, $\rho\in \{0,{\bf i},{\bf j},{\bf k}\}$ horizontal to the fibration 
$\text{Sp(2)}\to \Sigma^{7}_{GM}$ 
and taking values in ${H}^{G}$  such that 
\begin{equation*}
\big{(}dL_{p}\big{)}_{Id}\big{(}Ad_{p^{-1}}(u_{\rho}({\bf i}))\big{)}=\tilde{Y}^{\rho}_{p} \hspace{3ex} \mbox{for} \hspace{3ex} \rho\in \big{\{}0,{\bf i},{\bf j},{\bf k}\big{\}}.
\end{equation*}
Here $u_{0}({\bf i})=u_{0}, u_{\bf i}({\bf i})=u_{\bf i},  u_{\bf j}({\bf i})=u_{\bf j}, u_{\bf k}({\bf i})=u_{\bf k}$
are the matrices in $Ad_{p}(\mathfrak{h}_{p})$ defined in (\ref{base 0}) - (\ref{base k}) with $v={\bf i}$. Let 
{$$Y^{\rho}:=(d\pi_{GM})(\tilde{Y}^{\rho}) \hspace{3ex} \mbox{with} \hspace{3ex} \rho\in \big{\{}0,{\bf i},{\bf j},{\bf k}\big{\}} $$}
be the horizontal vector fields around $q=\pi_{GM}(p)$ taking values in $\mathcal{H}^{\Sigma}$. 
 \begin{prop}\label{final prop I-b}
The seven tangent vectors at $q=\pi_{GM}(p)$
\begin{align*}
&Y^{0}_{q},\,Y^{\bf i}_{q},\,Y^{\bf j}_{q},\,Y^{\bf k}_{q},\\
&[(d\pi_{GM})(\tilde{Y}^{0}),\,(d\pi_{GM})(\tilde{Y}^{\bf i})]_{q}
=[Y^{0},\,Y^{\bf i}]_{q},\\
&[(d\pi_{GM})(\tilde{Y}^{0}),\,(d\pi_{GM})(\tilde{Y}^{\bf j})]_{q}=[Y^{0},\,Y^{\bf j}]_{q},\\
&[(d\pi_{GM})(\tilde{Y}^{0}),\,(d\pi_{GM})(\tilde{Y}^{\bf k})]_{q}=
[Y^{0},\,Y^{\bf k}]_{q}
\end{align*}
span the tangent space $T_{q}(\Sigma^{7}_{GM})$.
\end{prop}
\begin{proof}
We define $W_{i}\in \mathfrak{sp}(2)$ ($i=1,2,3$) through the relations 
\begin{align*}
&(dL_{p})_{Id}\bigr(W_{1}\bigr)=[\tilde{Y}^{0},\,\tilde{Y}^{\bf i}]_{p},\\
&(dL_{p})_{Id}\bigr(W_{2}\bigr)=[\tilde{Y}^{0},\,\tilde{Y}^{\bf j}]_{p},\\
&(dL_{p})_{Id}\bigr(W_{3}\bigr)=[\tilde{Y}^{0},\,\tilde{Y}^{\bf k}]_{p},
\end{align*}
\par 
By the Lemmas \ref{basic lemma GM 2} and \ref{basic lemma GM 1} together with the explicit expressions of the commutators 
$$[u_{0}({\bf i}),u_{\bf i}({\bf i})]=[u_{0},\,u_{\bf}], \hspace{1ex}
[u_{0}({\bf i}),u_{\bf j}({\bf i})]=[u_{0},\,u_{\bf k}], \hspace{1ex}\mbox{and} \hspace{1ex} 
[u_{0}({\bf i}),u_{\bf k}({\bf i})]=[u_{0},\,u_{\bf k}] $$
{in (\ref{List_of_commutators}) and the identities in (\ref{Calculations_Adjoint}) we obtain: }
 \begin{align*}
 0&=\text{Tr}\big{(}Ad_{p}(W_{1})+[u_{0},u_{\bf i}]\big{)}\\
 &=\text{Tr}\big{(}Ad_{p}(W_{1})\big{)} =\text{Tr}\big{(}Ad_{p}(W_{1})-[u_{0},u_{\bf i}]\big{)}, \\
0&=<\lambda^{+}, W_{1}-Ad_{p^{-1}}([u_{0},\,u_{\bf i}])>,\\
0&=\text{Tr}\big{(}Ad_{p}(W_{2})+[u_{0},u_{\bf j}]\big{)}
=\text{Tr}\big{(}Ad_{p}(W_{2})-(-[u_{0},u_{\bf j}])\big{)},\\
0&=<\lambda^{+}, W_{2}-Ad_{p^{-1}}([u_{0},u_{\bf j}])> =<\lambda^{+}, W_{2}>\\
& =<\lambda^{+}, W_{2}-(Ad_{p^{-1}}(-[u_{0},u_{\bf j}]))>,\\
0&=\text{Tr}\big{(}Ad_{p}(W_{3})+[u_{0},u_{\bf k}]\big{)} =
\text{Tr}\big{(}Ad_{p}(W_{3})-(-[u_{0},u_{\bf k}])\big{)},\\
0&=<\lambda^{+}, W_{3}-Ad_{p^{-1}}([u_{0},u_{\bf k}])>=<\lambda^{+}, W_{3}>\\
& =<\lambda^{+}, W_{3}-\bigr(Ad_{p^{-1}}(-[u_{0},u_{\bf k}])\bigr)>.
\end{align*}

Hence we obtain
\[
Ad_{p}(W_{1})-[u_{0},\,u_{\bf i}],
Ad_{p}(W_{2})+[u_{0},\,u_{\bf j}],
Ad_{p}(W_{3})+[u_{0},\,u_{\bf k}]\in Ad_{p}(\mathfrak{h_{p}}),
\]
Analogously to the last case  (I-a) this implies that $\{Y^{\rho}\: | \: \rho=0,{\bf i}, {\bf j}, {\bf k}\}$ together with \[
(d\pi_{GM})_{p}\circ(dL_{p})_{Id}(W_{i})=(d\pi_{GM})_{p}([\tilde{Y}^{0},\,\tilde{Y}^{\rho}])
=[Y^{0},\,Y^{\rho}]_{q}, \rho \in \big{\{}0,{\bf i}, {\bf j}, {\bf k} \big{\}}
\]
span the tangent space $T_{q}(\Sigma^{7}_{GM})$. 
\end{proof}
It would be interesting to decide whether some or all of the remaining 26 exotic 7 spheres admit a higher co-dimensional sub-Riemannian structure. 
According to \cite{DPR,KZ,T} the Gromoll-Meyer exotic sphere  $\Sigma_{GM}^7$ is the only exotic sphere that is modeled by a biquotient of a compact group. 
If an exotic sphere is realized as a base space of a principal bundle in which the total space is not a group the method of the present paper are not applicable and 
new strategies are required for attacking this question. 
\providecommand{\bysame}{\leavevmode\hbox to3em{\hrulefill}\thinspace}

\end{document}